\documentclass[a4paper,11pt]{amsart}
\usepackage{hyperref}
\usepackage{geometry}
\usepackage{amsmath}
\allowdisplaybreaks
\usepackage{wrapfig}
\usepackage{amsmath,amsthm}

\usepackage{yfonts}
\usepackage[utf8]{inputenc}
\usepackage{amssymb}
\usepackage{graphics}
\usepackage{dirtytalk}
\usepackage{amstext}
\usepackage{subfigure}
\usepackage{engrec}
\usepackage{floatflt}
\usepackage{rotating}
\usepackage[font={footnotesize}]{caption}
\usepackage{framed}
\usepackage{listings}

\usepackage[normalem]{ulem}

\usepackage[titletoc,title]{appendix}

\newcommand{\mc}{\mathcal}
\newcommand{\mb}{\mathbb}

\newcommand{\R}{\mb R}

\newcommand{\N}{\mb N}
\newcommand{\Z}{\mb Z}

\newcommand{\T}{\mb T}

\newcommand{\D}{\mathrm{D}}

\newcommand{\eea}{\end{align}}

\renewcommand{\epsilon}{\varepsilon}
\renewcommand{\bar}{\overline}
\renewcommand{\tilde}{\widetilde}

\newcommand{\bo}{\boldsymbol}
\renewcommand{\phi}{\varphi}

\DeclareMathOperator{\Leb}{Leb}

\DeclareMathOperator{\id}{id}
\DeclareMathOperator{\diam}{diam}

\renewcommand\upsilon{\theta}

\newtheorem{theorem}{Theorem}[section]

\newtheorem{corollary}{Corollary}[section]
\newtheorem{lemma}{Lemma}[section]

\newtheorem{proposition}{Proposition}[section]
\theoremstyle{definition}
\newtheorem{definition}{Definition}[section]

\newtheorem{assumptionx}{Assumption}
\newenvironment{assumption}[1]
{\renewcommand\theassumptionx{#1}\assumptionx}
{\endassumptionx}

\newtheorem{remark}{Remark}[section]
\newtheorem{example}{Example}[section]

\newtheoremstyle{algorithm}
{4pt}
{4pt}
{}
{}
{}
{:}
{\newline}
{}

\usepackage{xcolor}

\newtheorem{algorithm}{Algorithm}

\newcommand{\balgorithm}{\begin{algorithm}\begin{framed}\ }
\newcommand{\ealgorithm}{\end{framed}\end{algorithm}}

\newcommand{\pippo}{\begin{code} \  \begin{lstlisting}[frame=single]}
\newcommand{\pappo}{\end{lstlisting} \end{code}}

\newcommand{\bd}{\begin{definition}}
\newcommand{\ed}{\end{definition}}

\newcommand{\bt}{\begin{theorem}}
\newcommand{\et}{\end{theorem}}
\newcommand{\bp}{\begin{proposition}}
\newcommand{\ep}{\end{proposition}}
\newcommand{\bc}{\begin{corollary}}
\newcommand{\ec}{\end{corollary}} 
\newcommand{\bl}{\begin{lemma}}
\newcommand{\el}{\end{lemma}}
\newcommand{\br}{\begin{remark}}
\newcommand{\er}{\end{remark}}

\DeclareMathOperator{\Lip}{Lip}

\usepackage{enumitem}

\title[Random-like properties of deterministic skew-products]{ Random-like properties of chaotic forcing}
\date{}

\author{Paolo Giulietti}
\address{Paolo Giulietti: Centro di Ricerca Matematica 
Ennio de Giorgi, Scuola Normale Superiore, Piazza dei 
Ca\-valieri 7, 56126 Pisa, Italy.
E-mail: \texttt{paolo.giulietti@sns.it}}
\author{Stefano Marmi}
\address{Stefano Marmi: Scuola Normale Superiore, Piazza dei 
Ca\-valieri 7, 56126 Pisa, Italy \\
E-mail: \texttt{stefano.marmi@sns.it}}
\author{Matteo Tanzi}
\address{ Matteo Tanzi: Courant Institute of Mathematical Sciences, New York University, New York, NY 10012, USA\\
E-mail: \texttt{matteo.tanzi@nyu.edu}}
\thanks{MT acknowledges funding from the H2020 Marie Sk{\l}odowska-Curie Actions  ``Ergodic Theory of Complex Systems" project no. 843880. He is also grateful for the hospitality of  the Centro di Ricerca Matematica Ennio de Giorgi and Scuola Normale Superiore where part of this work was carried out. }

\begin{document}

\maketitle

\begin{abstract}
 We prove that skew systems with a sufficiently expanding base have approximate exponential decay of correlations, meaning that the exponential rate is observed modulo an error. The fiber maps are only assumed to be Lipschitz regular and to depend on the base in a way  that guarantees diffusive behaviour on the vertical component. The assumptions do not imply an hyperbolic picture {and one cannot rely on the spectral properties of the transfer operators involved}. The approximate nature of the result is the inevitable price one pays for having so {mild assumptions on the dynamics on the vertical component}. However, the error in the approximation goes to zero when the expansion {of the base} tends to infinity.  The result can be applied beyond the original setup when combined with acceleration or conjugation arguments, as our examples show.
\end{abstract}

\section{Introduction}

One of the main questions of modern dynamical systems theory is: \emph{to which extent a deterministic chaotic system resembles a  random process?} This question has been addressed in various contexts from different point of views (see  \cite{Young19} for a review). Here we study it in relation to forcing, and in particular we investigate the similarities  between random  and (sufficiently chaotic) deterministic forcing focusing on the statistical properties of the forced system.

A forced system is a system whose intrinsic dynamics is affected by an external  influence  typically coming from  the interaction with another system or the surrounding environment. The forcing can be modelled to be random, e.g. obtained by adding to the dynamics a noise term independent in time, or deterministic, i.e. dependent on a variable  that evolves in time following a deterministic law\footnote{For precise definitions and a comparison between deterministic and random forcing see Section \ref{App:RandSysMarkChains} in the Appendix.}.

In the random case,  classical results from the theory of Markov chains  show that if there is enough diffusion, e.g. if the forcing adds smooth unbounded noise to the dynamics, then the forced system has a stationary measure  that describes its asymptotic statistical behaviour, and exhibits memory loss and annealed exponential decay of correlations (among others \cite{ DownMeynTweedie, BaladiYoung}).  In contrast, if  the forcing is deterministic, it is well known that even just to prove  existence of a physically relevant invariant measure one needs to impose strong assumptions both on the intrinsic dynamic and on the forcing, often leading to some degree of hyperbolicity of the system and/or a good spectral picture of the operators involved (see literature below).     

 In this paper we prove that,  if the forcing  has a ``diffusive effect" and  is generated by a uniformly expanding map with  high expansion,  then the deterministic  system has an \emph{approximate} stationary measure and exhibits \emph{approximate} decay of correlations.  We postpone rigorous definitions to later sections. Loosely speaking, an approximate stationary measure describes the asymptotic statistical properties of the system modulo a controlled error, and by an approximate exponential decay of correlations we mean that measurements of observables along orbits exhibit exponential decay of correlations  also modulo an error. Most importantly, these errors 
 go to zero when the expansion of  
 the map generating the forcing goes to infinity. In other words we could say that, when the expansion of the map generating the forcing goes to infinity, the deterministic forcing becomes indistinguishable from random forcing with respect to the statistical properties we analyze.
 
 {It's important to remark that our requirements do not ensure global hyperbolic properties or a good spectral picture, and even the existence of a {physically relevant} invariant measure cannot be deduced from the assumptions. The price that we pay is the {approximate} nature of the result. Its relevance, however, is clear} when having an eye to applications; {here decorrelation estimates come from observations of real-world systems and are intrinsically affected by a measurement error:  if this error} is {larger} than the approximation error in the decorrelation estimate, exact and approximate decay of correlations are indistinguishable.
 
{Our approach is quite flexible and we expect it to be adaptable to a variety of situations beyond the current working assumptions, for example in situations with lower regularity, or in combination with various conjugations arguments (see Section \ref{Sec:Generalizations} for some generalizations). }
\subsection{Literature review}
 In {mathematical} terms, a forced system in discrete time can be described by a \emph{skew-product} transformation which is a map $F:\Omega\times X\rightarrow \Omega\times X$ such that 
\begin{equation}\label{Eq:DefFunctionF}
F(\omega,x)=(g(\omega),f(\omega,x))
\end{equation}
where $g:\Omega\rightarrow \Omega$ and $f:\Omega\times X\rightarrow X$. The set $\Omega$ is called the \emph{base} of the skew-product, while $X$ is referred to as the \emph{vertical fiber}. The main characteristic of a skew-product is that the evolution on the vertical fiber $X$ depends on the state of the base $\Omega$, but not vice versa.  
 
 The literature on skew-products is vast to the extent that there are entire research trends studying particular aspects of these systems (e.g. iterated function systems, random dynamical systems, smoothness of invariant graphs over skew-products, etc.). Here we focus on those works dealing with statistical properties of skew products that have a ``deterministic" base, such as   \cite{gouezel2007statistical, stenlund2011non, santos2013distributional,GRS, ButterleyEslami,   B18, dragivcevic2018spectral, nicol2018central, walkden2018invariant, hafouta2019, Kloeckner18, dragivcevic2020spectral} and references therein.  These works usually only require $g$ to be a measure preserving ergodic transformation or, at most, to exhibit some uniform hyperbolicity. However, they restrict the fiber map $f$ to one of some particular classes to  ensure contraction or hyperbolic properties (exact or averaged) of the vertical fiber.    In contrast, our results make only mild regularity assumptions on $f$, but require that $g$ is uniformly expanding  with large minimal expansion.
 
 As a consequence of our requirements, the map $F$ is likely to have a dominated splitting of the tangent space and be \emph{partially hyperbolic}  (see e.g.  \cite{Pesin, Sambarino13}) with an expanding direction roughly aligned with the base  dominating the other invariant directions.  To put our work under this perspective, let us remind that available results on existence of physical measures and decay of correlations {for partially hyperbolic systems} often assume low dimensional geometry either of the phase space or of some invariant directions, and/or nonvanishing Lyapunov exponents (\cite{BV00, ABV00, Dima04A, Dima04B, Tsujii05, alves2017srb, tanzi2020nonuniformly}) which, in general, are not granted in our setup.  More recent results give sufficient conditions for partially hyperbolic systems to have exponential decay of correlations by  turning qualitative topological conditions {such as accessibility (\cite{WB2010}), into quantitative properties of the operators involved} (\cite{CastorriniLiverani20, GRH20}). {The systems we consider do not fit in these results due to lack of smoothness,  but it is unclear if the assumptions can be verified even for those systems in our setup which have the required regularity.} 
 
 As the base map is much more chaotic than the vertical fibers, our setup is reminiscent of fast-slow systems (see  \cite{DeSimoiLiverani18, CastorriniLiverani20, CFKM20, KorKoMel} among many others). However,  the dynamic of our skew-products {does not present separation of time-scales since at each time step} it can produce displacements of the same order in both the base system and the vertical fibers.

\subsection{Organization of the paper}
In Section \ref{sec:setting} we present the setting,  the results, some examples, a sketch of the proof. In Section \ref{Sec:NoDistortion} we prove our result in the simpler situation where the map in the base has no distortion and the phase space is 2D. In Section \ref{Sec:Distortion} we prove our main theorem  in full generality. In Section \ref{Sec:Generalizations} we discuss some generalizations. In the appendices we gather some background material and results on Markov chains (in Appendix \ref{App:MarkChainsAndRandom}), disintegration of measures (in Appendix \ref{App:RohlinThm}), and some computations involving the Kantorovich-Wasserstein distance that are used throughout the proofs (in Appendix \ref{App:Wasserstein}).

\section{Setting and Results} \label{sec:setting}

\subsection{ Setting}
Let's consider a map $F$ as in \eqref{Eq:DefFunctionF} where we set $\Omega=\T^{m_1}$ and $X=\T^{m_2}$, here $\T=\R/ \Z$ is the 1D torus and $m_1 , m_2 $ two positive integers. In the following we will denote by $|p_1-p_2|$ the distance between $p_1,p_2\in\T^{N}$ regardless of the specific $N\in\N$. For $I \subseteq \T^{m}$ be a set we denote by $\mbox{Op}(I)$ its open part.

\subsubsection{ The base map $g$} Consider  $g:\T^{m_1}\rightarrow \T^{m_1}$ a  $C^2$ local diffeomorphism. In particular,  there is $d\in\N$ and $\mc I=\{I_i\}_{i=1}^d$ a partition of $\T^{m_1}$ such that: $\mbox{Op}(I_i)=I_i \mod 0$, $\{g_i:=g|_{I_i}\}_{i=1}^d$ with $g_i:I_i\rightarrow \T^{m_1}$ are invertible  branches of $g$, and $g_i|_{\mbox{Op}(I_i)}$ is $C^2$. Call $\{h_i:=g_i^{-1}\}_{i=1}^d$  the corresponding inverses.  

We assume that $g$ satisfies the following assumptions:
\begin{equation}  \label{As:0.1}
\exists \sigma>1\,\,\,\mbox{s.t. } \|   \D g_\omega v\| \ge \sigma \|v\|\quad\quad\forall \omega\in\T^{m_1},\,v\in\R^{m_1}, 
\tag{H0.1} \end{equation}
where $\|\cdot\|$ is the Euclidean norm on $\R^{m_1}$, and
\begin{equation}  \label{As:0.2}
\exists D>0\,\,\, \mbox{s.t. }\frac{| \D g_{h_i(\omega_1)}|}{|\D g_{h_i(\omega_2)}|}\le e^{D|\omega_1-\omega_2|} \quad\quad\forall \omega_1,\omega_2\in \T^{m_1}\mbox{ and }\forall i. 
\tag{H0.2} \end{equation}
where $| \D g_{h_i(\omega_1)}|$ denotes the determinant of $\D g_{h_i(\omega_1)}$. Condition \eqref{As:0.1} states that the differential of $g$ expands vectors in tangent space of  a factor at least $\sigma>1$, while \eqref{As:0.2} imposes a uniform bound on the distortion.  It is well known that  $g$ has a unique absolutely continuous invariant probability (a.c.i.p.) measure (see \cite{BG97},\cite{SDDSViana} and references therein). We call $\nu_g$ this measure and $\rho_g:=\frac{d\nu_g}{d\Leb_{\T^{m_1}}}$ its density, where $\Leb_{\T^{m_1}}$ is the Lebesgue measure on $\T^{m_1}$.

 \subsubsection{ The vertical fiber maps $f$}We assume  $f:\T^{m_1}\times \T^{m_2}\rightarrow \T^{m_2}$ to be at least Lipschitz, and denote by $L\ge 0$ the Lipschitz constant, namely
 \begin{equation}\label{As:0.3}
 L:=\inf_{(\omega_1,x_1)\neq(\omega_2,x_2) } \frac{|f(\omega_1,x_1)-f(\omega_2,x_2)|}{ |(\omega_1,x_1)-(\omega_2,x_2)|}.\tag{H0.3}
 \end{equation}
 
 Let $\{f_\omega\}_{\omega\in\T^{m_1}}$ be the collection of maps $f_\omega:\mb T^{m_2}\rightarrow\T^{m_2}$ i.e.  $f_\omega(\cdot):=f(\omega,\cdot)$.  We  write $f(\cdot,x)$ for the maps $f(\cdot,x):\T^{m_1}\rightarrow \T^{m_2}$ obtained by fixing $x\in\T^{m_2}$ and letting $\omega\in\T^{m_1}$ vary. We let $\pi_1:\mb T^{m_1}\times \mb T^{m_2}\rightarrow \T^{m_1}$ be the projection onto the \emph{horizontal} $\T^{m_1}$-coordinate  and, given a measure $\mu$ on $\mb T^{m_1}\times \T^{m_2}$,  we refer to $\pi_{1*}\mu$ as the \emph{horizontal marginal} of $\mu$. We also denote by $\pi_2:\mb T^{m_1}\times \mb T^{m_2}\rightarrow \T^{m_2}$ the projection onto the \emph{vertical} $\T^{m_2}$-coordinate and refer to $\Pi\mu:={\pi_{2*}}\mu$ as the \emph{vertical marginal} of the measure $\mu$.

\subsubsection{$\mc P$, the random counterpart of $F$}In the following,  $\mc M_1(Y)$ denotes the space of Borel probability measures on the compact metric space $Y$. 

For $\mu\in\mc M_1(\T^{m_2})$, define the push-forward $f_{\omega*}\mu(A) = \mu( f_\omega^{-1}(A)) $ for any measurable $ A \subseteq \T^{m_2}$,  and define the operator  $\mc P:\mc{M}_1(\T^{m_2})\rightarrow \mc{M}_1(\T^{m_2})$
\begin{equation}\label{Def:OperatorP}
\mc P \mu:=\int_{\T^{m_1}}  d\nu_g(\omega)f_{\omega*}\mu\,=\int_{\T^{m_1}}d\omega \rho_g(\omega)f_{\omega*}\mu.
\end{equation}
Notice that  $\mc P$ is the generator for a \emph{discrete time stationary Markov process} with transition kernel
\[
P(x,A):=\int_{\T^{m_1}}{\delta}_{f_{\omega}(x)}(A)\rho_g(\omega)d\omega.
\]  
where $\delta_{f_\omega(x)}$ denotes the Dirac mass at $f_\omega(x)$. These operators are well studied in  the literature and sufficient conditions under which $\mc P$ has a spectral gap in various functional spaces are known {(see e.g. \cite{HM08, Sh95, Stroock14} and Appendix \ref{App:MarkChainsAndRandom})}.

It is important to notice that if at each time step one was to apply a map  $\{f_{\omega}\}_{\omega\in\T^{m_1}}$  sampled \emph{independently} with respect to $\nu_g$, then the operator $\mc P$ would describe the evolution of the vertical marginal. In other terms, one can think of the Markov chain generated by $\mc P$ as the \say{random counterpart} of  the deterministic evolution given by  $F$ which instead {selects} the map $f_\omega$ at each time-step  according to the deterministic process  $\omega, \,g(\omega),\, g^2(\omega),...$ generated by $g$.

\subsection{Main Assumption}
 {Assumption \ref{As:1} below requires that the Markov chain generated by $\mc P$ is geometrically ergodic   with respect to the Total Variation (TV) distance (see Appendix \ref{App:MarkChainsAndRandom} for definitions). }

\begin{assumption}{H}  \label{As:1}

There are $C>0$, $\lambda\in(0,1)$ such that 
\[
 d_{TV}(\mc P^{n}\mu,\mc P^{n} \nu)\le C\lambda^nd_{TV}(\mu,\nu),
\]
for all $\mu,\,\nu\in\mc M_1(\T^{m_2})$.
\end{assumption}

Notice that by a Krylov-Bogolyubov argument, it follows that there is a unique $\eta_0\in\mc M_1(\T^{m_2})$ invariant under $\mc P$, i.e. such that  $\mc P\eta_0=\eta_0$, which is called a \emph{stationary measure} for the Markov process generated by $\mc P$. Also notice that Assumption \ref{As:1} is a condition on $\mc P$, and therefore it {depends  on   $f:\T^{m_1}\times\T^{m_2}\rightarrow \T^{m_2}$ and 
 $\nu_g$ only.}

\subsection{Main Result}

When describing the statistical properties of a skew-product such as $F$, we adopt the following point of view. We assume to have access to observations  of measurable functions $\phi:\T^{m_2}\rightarrow \R$ along the orbits of the system.   Picking as reference measure on $\T^{m_1}\times \T^{m_2}$  the Lebesgue measure $\Leb_{\T^{m_1}\times\T^{m_2}}$ gives rise to the sequence of dependent random variables
\[
\left\{\phi\circ \pi_2 \circ F^n \right\}_{n=1}^{+\infty}
\] 
on $(\T^{m_1}\times \T^{m_2},\Leb_{\T^{m_1}\times\T^{m_2}})$. 

For $\phi,\psi:\T^{m_2}\rightarrow \R$ in suitable functional spaces, we ask if there are constants $A\in\R$ and $ \tilde \lambda\in (0,1)$ such that 
\begin{equation}\label{Eq:AnnealedDecay}
\left |\int_{\T^{m_1}\times\T^{m_2}}\phi( \pi_2F^n(\omega,x)) \psi(x)\,\,d\omega dx-A\right|=\mc O( \tilde \lambda^n)
\end{equation}
When \eqref{Eq:AnnealedDecay} is satisfied, the system is said to have exponential annealed decay of correlations. The term \emph{annealed} refers to the fact that the observables $\phi,\psi$ depend on the vertical $\T^{m_2}$-coordinate only, and therefore the correlations are averaged with respect to the horizontal $\T^{m_1}$-coordinate.

 As already argued in the introduction, our systems have little hope to satisfy \eqref{Eq:AnnealedDecay},  but the following theorem shows that $F$ exhibits exponential annealed decay of correlations, up to a given precision that depends on the expansion of the base system. 

\begin{theorem}\label{Thm:AppDecayofCorr}
Let $F$ satisfy assumptions \eqref{As:0.1}-\eqref{As:0.3} and Assumption \eqref{As:1} with datum $m_1$, $m_2\in\N$, $D,\,L,\, C>0$, $\sigma > 1$, $ \lambda\in(0,1)$. For every   $\epsilon>0$ there is  $\sigma_0 > \max\{1,L\} $ (depending on $\epsilon$ and all the datum but $\sigma$) such that  if  $\sigma > \sigma_0 $, then  there are $\tilde\lambda\in(0,1)$, $\tilde C>0$  and a probability measure ${\bar \eta}\in\mc M_1(\T^{m_2})$ such that 
\begin{align*}
&\left|\int_{\T^{m_1}\times\T^{m_2}}\phi( \pi_2F^n(\omega,x)) \psi(x)\,\,d\omega dx-\int_{\T^{m_2}} \phi(x)d{\bar \eta}(x)\int_{\T^{m_2}}\psi(x)dx\right|\leq C_{\phi,\psi}(\tilde C \tilde{\lambda}^n+ \epsilon)
\end{align*}
for all $\psi \in L^1(\T^{m_2};\R)$ and $\phi\in\Lip(\T^{m_2};\R)$ where $C_{\phi,\psi}>0$ depends on $\phi,\psi$ but not from $n, \epsilon$. 

\end{theorem}

In fact, we will prove something stronger. Loosely speaking, we show that under the assumptions of Theorem \ref{Thm:AppDecayofCorr}, for any $\mu\in\mc M_1(\T^{m_1}\times\T^{m_2})$ which is sufficiently regular, in a sense that will be specified  below,  the distance\footnote{Here the distance is with respect to the Wasserstein-Kantorovich metric defined in equation \eqref{Eq:KantWassDist} below.} between the vertical marginal of $F_*^n\mu$ and $\bar \eta$ can be upper bounded by $O(\tilde \lambda^n+\epsilon)$ (see Proposition \ref{Prop:KantDistanceDistCase} for a rigorous statement). We call this phenomenon \emph{approximate memory loss}. For a  definition and an example of (exact) memory loss see e.g. \cite{OSY09}.

{The measure ${\bar\eta}$ above plays the role of  an \emph{approximate stationary measure} for the forced system}.  In the case with no distortion, e.g. $g(\omega)=\sigma\omega$ mod 1 with $\sigma\ge 2$, ${\bar\eta}$ equals $\eta_0$, the {stationary measure} of $\mc P$. As shown in Section \ref{sec:difffixedpooint}, when there is distortion,  ${\bar\eta}$ can be different from $\eta_0$, and is related to the fixed point of another operator, called $	\bo{\mc L}$, introduced in Section \ref{Sec:TrackingDistVertMarg}.

\begin{remark} \noindent
\begin{itemize}
\item Given $D$ and $L$, one might need a large minimal expansion $\sigma_0$ to ensure that $\epsilon>0$ is small. Examples of base maps $g$ with given distortion, and arbitrarily large minimal expansions $\sigma_0$ can be constructed easily by fixing any map $g_0:\T^{m_1}\rightarrow \T^{m_1}$ satisfying \eqref{As:0.1}-\eqref{As:0.2}, and considering $g:=g_0^n$ with high $n\in\N$.  With this definition, $g$ has minimal expansion equal to the minimal expansion of $g$ raised to the power $n\in\N$, and distortion  uniformly bounded with respect to $n$.
 \item Existence of an invariant measure which is physical  or with some smoothness such as an SRB measure (see \cite{Young02} for definitions)  has little hope in general. {One reason is the low regularity of $F$ which is only Lipschitz. However,} imposing higher regularity, e.g. $F$ globally $C^{1 + \alpha}$, would not be enough {as the domination that (possibly) results from the high expansion in the base, even if it can lead to existence of positive Lyapunov exponents, cannot ensure existence of an SRB or physical measure by itself, and all the more reasons not to expect exact exponential decay of correlations. }
\item {We can give an explicit bound for the constant $C_{\phi,\psi}$}. Letting $\psi-\int_{\T^{m_2}}\psi=\psi_1-\psi_2$ with $\psi_1,\psi_2\ge 0$ being the positive and negative components of $\psi-\int_{\T^{m_2}}\psi$,
\[
C_{\phi,\psi}\le  2\|\psi\|_{L^1}(\Lip(\phi)+1).
\]
\item As mentioned in the introduction, whenever one has additional information on the fiber maps $\{f_\omega\}_{\omega\in\T}$, other approaches could lead to more precise statements. 
\end{itemize}
\end{remark}

\subsection{Examples}\label{Sec:Examples}
One way to ensure that Assumption \ref{As:1} holds is by imposing  two main regularity requirements on $f$ with respect to the horizontal variable $\omega$, i.e. with respect to the forcing: 1) \emph{Regularity condition:} $f$ is $C^k$ in the variable $\omega$ for a sufficiently large $k$. 2) \emph{Non-degeneracy condition:} the differential of $f$ with respect to $\omega$ is invertible which, for every $x\in\T^{m_2}$, makes the function $f(\cdot, x):\T^{m_1}\rightarrow \T^{m_2}$ a local diffeomorphism  on its range (notice that for this requirement to hold $m_1$ has to be equal to $m_2$).

\begin{example} \label{Ex:assexp}
Let's consider $m_1=m_2=m$, and assume   that  for any $x\in \T^{m}$ $f(\cdot,x):\T^{m}\rightarrow \T^m$ is a $C^2$ local diffeomorphism {or, equivalently, $\{f_\omega\}_{\omega\in\T^m}$ is a family of maps with $C^2$ dependence on the parameter $\omega$  such that the differential $(\D f(\cdot,x))_\omega$ is bijective for every $x,\omega\in\T^{m}$. }

From Eq. \eqref{Def:OperatorP} one can deduce that 
\[
\mc P \delta_x=f(\cdot,x)_*\nu_g
\]
and  since $f(\cdot,x)$ is  a non-singular transformation, the expression of its Perron-Frobenius operator gives
\[
\frac{d\mc P\delta_x}{d\Leb_{\T^{m}}}(y)=\sum_{k}\frac{\rho_g(y_k)}{(| \D f(\cdot,x)|)_{y_k}}
\]
where the sum is over all the preimages $y_k$ of $y$ under the map $f(\cdot,x)$. $\frac{d\mc P\delta_x}{d\Leb_{\T^{m}}}$ is in $C^1$ since $| \D f(\cdot,x)|$ and $\rho_g$ are $C^1$ functions. It is also uniformly bounded away from zero, as there is $c_1>0$ such that $\rho_g>c_1$ (see e.g. \cite{SDDSViana}), and there is $K_1>0$ such that   $|(\D f(\cdot,x))_\omega|\le K_1$ for every $\omega,x\in \T^m$.  This implies that for every $x\in\T^{m}$, $\frac{d\mc P\delta_x}{d\Leb_{\T^{m}}}(y)>cK_1$, i.e.  the densities of the transition probabilities are all uniformly bounded away from zero. It is well known that the Markov chain generated by $\mc P$ is geometrically ergodic, i.e. satisfies Assumption \eqref{As:1} (see  Theorem \ref{Thm:GeometricErgodicity} in the Appendix).
\end{example}

The following  example is a subcase of the example above and  shows one of the simplest possible nontrivial setups.

\begin{example}[System with additive deterministic noise]
$f(\omega,x) = T(x) + h(\omega)$, where $T:\T^m\rightarrow\T^m$ is any Lipschitz map, and $h:\T^m\rightarrow\T^m$ is a local $C^2$ diffeomorphism. 
\end{example}
Let us  stress that these sufficient conditions for Assumption \ref{As:1}: 1) are by no means necessary;  2) give no control on a single fiber map $f_\omega$ beyond the requirement that it is Lipschitz regular;  3) do not imply  good spectral properties for $F_*$.   

\subsection{Sketch of the proof}  

Given two probability measures $\mu_1,\mu_2$ on $\T^{m_2}$, the Kantorovich-Wasserstein distance between them is defined as
\begin{equation}\label{Eq:KantWassDist}
d_W(\mu_1,\mu_2):=\inf_{\gamma\in\mc C(\mu_1,\mu_2)}\int_{\T^{m_2}\times \T^{m_2}}\mathtt |x-y|d\gamma(x,y)
\end{equation}
where $\mc C(\mu_1,\mu_2)$ is the set of couplings of $\mu_1$ and $\mu_2$, i.e. the set of all probability measures on $\T^{m_2}\times \T^{m_2}$ with marginals $\mu_1$ and $\mu_2$ respectively on the first and second factor. 

The class of measures defined below plays a central role in the  proof of Theorem \ref{Thm:AppDecayofCorr}.
\begin{definition} \label{Def:HolderDisint}
Given a probability measure $\mu$ on $\T^{m_1}\times \T^{m_2}$, we say that $\mu$ has \emph{Lipschitz disintegration along vertical fibres}, or simply \emph{Lipschitz disintegration}, if there is a disintegration $\{\mu_\omega\}_{\omega\in \T^{m_1}}$  of $\mu$, with respect to the measurable partition $\{\{\omega\}\times \T^{m_2}\}_{\omega\in\T^{m_1}}$, such that the map $\omega\mapsto \mu_\omega$ from $(\T^{m_1},|\cdot|)$ to $(\mc M_1(\T^{m_2}),  d_W)$ is Lipschitz.  Let
\[
\Lip(\mu):=\inf_{\omega_1\neq\omega_2}\frac{d_W(\mu_{\omega_1},\mu_{\omega_2})}{|\omega_1-\omega_2|}
\]
the Lispchitz constant of  $\omega\mapsto \mu_\omega$ 
\end{definition}
 Since $\mc M_1(\T)$ with the metric $d_W$ is complete, if $\mu$ admits a Lipschitz disintegration, this disintegration is unique and $\Lip(\mu)$ is well defined. \footnote{In Appendix \ref{App:RohlinThm} we gather statements, such as the above, about disintegration of measures that will be used throughout the paper. }

To prove  {Theorem \ref{Thm:AppDecayofCorr}},  we are going to study the evolution of probability measures on $\T^{m_1}\times\T^{m_2}$ that have a Lipschitz disintegration {with a focus on the evolution of their vertical marginals}. To do so we follow the steps below.
\begin{itemize} 
\item[1)] First of all we show that under the assumptions of the main theorem,  if $\mu\in\mc M_1(\T^{m_1}\times \T^{m_2})$ has Lipschitz disintegration, so does $F^n_*\mu$ for any $n\in\N$, and $\Lip(F_*^n\mu)$ is bounded uniformly in $n\in\N$ (see Proposition \ref{Prop:ControlHDist}). This is a consequence of the uniform (high) expansion of the map $g$. This result shows the existence of an invariant class of measures whose disintegration is smooth along the {$\T^{m_1}$-coordinate, and  is proved by} using an explicit expression for a disintegration of $F_*\mu$ in terms of a disintegration of $\mu$ (see Proposition \ref{Prop:FormDisint}).
\item[2)] Next, we use the above fact to show that the  vertical marginal  of $F^n_*\mu$ can be approximated by looking at the action of an auxiliary operator, $\bo{\mc L}$, that acts on a suitable decomposition of $\mu$, and that, unlike $F_*$, has {contraction properties that can be exploited} (see \eqref{Eq:OpL} for the definition of $\bo{\mc L}$,  Proposition \ref{Prop:ApproximDisitnMathcalL}, and Proposition \ref{Prop:SpecPropMCL}). 
\item[3)] The above approximation allows to show that under application of $F_*$, the system exhibits  \emph{approximate exponential memory loss} on its vertical marginal. By this we mean that given any two probability measures $\mu_1,\mu_2\in\mc M_1(\T^{m_1}\times \T^{m_2})$ with Lipschitz disintegration, the Kantorovich-Wasserstein distance between vertical marginals $\Pi F_*^n\mu_1$ and $\Pi F_*^n\mu_2$ shrinks exponentially fast modulo an approximation error (see Proposition \ref{Prop:KantDistanceDistCase}).
\item[4)]Finally, we use the above approximate memory loss to prove the approximate decay of correlations. 
\end{itemize}

In Section \ref{Sec:NoDistortion}, we  give a proof of Theorem \ref{Thm:AppDecayofCorr} in the simpler case where: $m_1=m_2=1$;  the dynamics in the base is smooth and has no distortion, i.e. $D=0$. Under these assumptions $g:\T\rightarrow\T$ can be written as
\begin{equation} \label{As:0plus}
g(\omega)=\sigma\omega\mod 1, \quad\quad \sigma\in\N\backslash\{1\}. \tag{A0+}  
\end{equation}
This is for the sake of presentation since in this case the treatment of points 2) and 3) does not require the introduction of the auxiliary operator $\mc L$, whose role is played by $\mc P$. This makes the  proof much easier than in the fully general case.  

\begin{remark}
Picking a metric on $\mc M_1(\T^{m_2})$ as weak as the Wasserstein metric $d_W$ is crucial to our arguments. Without further assumptions on $\{f_\omega\}_{\omega\in\Omega}$, measures with  Lipschitz disintegration with respect to the $d_{TV}$ distance, for example, may not constitute an invariant class {with respect to the action of $F_*$}.  
\end{remark}


%

\section{Case without distortion} \label{Sec:NoDistortion}
In this section we work under Assumption \ref{As:0plus}. Namely we consider  $g:\T\rightarrow \T$ defined as $g(\omega)=\sigma \omega \mod 1 $, where $\sigma \in\N\backslash\{1\}$. Recall that under these assumptions $\nu_g=\Leb_\T$ and  $\rho_g$ is constant equal to one. Take $\mu\in\mc M_1(\T\times \T)$ having  horizontal marginal $\Leb_\T$.  We study the evolution of $\mu$ under applications of  $F_*$. First of all notice that as a consequence of the product structure of $F$ and invariance of $\Leb_\T$ under $g$, also $F_*\mu$ has horizontal marginal equal to $\Leb_\T$. The evolution of the disintegration along vertical fibres, instead is described in the following proposition.
\begin{proposition}
Let $\mu$ be a probability measure on $\T\times \T$ with horizontal marginal equal to $\Leb_{\T}$. Let $\{\mu_\omega\}_{\omega\in\T}$ be a disintegration of $\mu$ along  vertical fibres, then a disintegration of $F_*\mu$ along  vertical fibres is given by $\{(F_*\mu)_\omega\}_{\omega\in\T}$ with
\begin{equation}\label{Eq:FormDisintNodist}
(F_*\mu)_{\omega}=\frac{1}{\sigma}\sum_{i=0}^{\sigma-1}f_{\frac{\omega+i}{\sigma}*}\mu_{\frac{\omega+i}{\sigma}}.
\end{equation}
\end{proposition}
\noindent
The proof
is a particular case of Proposition \ref{Prop:FormDisint}, therefore we omit the details. Sufficient to say that, in this setting one has, {for an interval $I$}
\[
(F_* \mu)_\omega(I)=\lim_{\delta\rightarrow 0}\frac{(F_* \mu)([\omega-\delta,\omega+\delta]\times I)}{2\delta }
\]
where the numerator can be easily controlled.  
Next, recall Definition \ref{Def:HolderDisint}. In the proposition below we use \eqref{Eq:FormDisintNodist} to deduce that if $\mu$ has Lipschitz disintegration, then so do all its iterates  $F_*^n\mu$, and if $\sigma$ is sufficiently large, then their Lipschitz constants are all uniformly bounded. 

Before moving to the next proposition we recall for the reader's conveninece a property of Wasserstein spaces (see e.g. \cite{Villani09} for details). Given a Borel signed measure $\xi$ on $\T$ with $\xi(\T)=0$,  consider the Wasserstein norm
\[
\|\xi\|_W:=\sup_{\phi\in\Lip^1(\T)}\, \int_\T \phi\, d\xi.
\] 
Recall that we denoted by $\Lip^1(\T):=\{\phi:\T\rightarrow \R:\,\Lip(\phi)\leq 1\}$  the Lipschitz functions on $(\T,|\cdot|)$ with Lipschitz constant less or equal to one (we write $\Lip^1$ when there is no risk of confusion). The Kantorovich-Wasserstein distance defined in \eqref{Eq:KantWassDist} can be rewritten as 
\[
d_W(\mu,\nu)=\|\mu-\nu\|_W=\sup_{\phi\in\Lip^1}\int_\T\phi(x) \,d(\mu-\nu)(x)
\] 
which will simplify the notation later in the proofs.

\begin{proposition} \label{Prop:LipNodist}
 Let $\mu$ be a probability measure on $\T\times\T$ with horizontal marginal $\Leb_\T$ and Lipschitz disintegration $\{\mu_\omega\}_{\omega\in\T}$. Then the disintegration  of $F_*\mu$ defined in Eq. \eqref{Eq:FormDisintNodist} is also Lipschitz and 
  \[
\Lip\left(F_*\mu\right)\le L\sigma^{-1}\Lip(\mu)+L\sigma^{-1}.
  \]
 \end{proposition} 
 \begin{proof}
 \begin{align*}
 d_W((F_*\mu)_\omega,(F_*\mu)_{\omega'})&=\sup_{\phi\in\Lip_{ }^1}\,\int_\T \phi\, d\left(\frac{1}{\sigma}\sum_{i=0}^{\sigma-1}f_{\frac{\omega+i}{\sigma}*}\mu_{\frac{\omega+i}{\sigma}}-\frac{1}{\sigma}\sum_{i=0}^{\sigma-1}f_{\frac{\omega'+i}{\sigma}*}\mu_{\frac{\omega'+i}{\sigma}}\right)\\
 &\le\frac{1}{\sigma}\sum_{i=0}^{\sigma-1}\,\,\sup_{\phi\in\Lip_{ }^1}\,\int_\T \phi d\left(f_{\frac{\omega+i}{\sigma}*}\mu_{\frac{\omega+i}{\sigma}}-f_{\frac{\omega'+i}{\sigma}*}\mu_{\frac{\omega'+i}{\sigma}}\right)
 \end{align*}
  Calling $\omega_i:=\frac{\omega+i}{\sigma}$, and $\omega_i':=\frac{\omega'+i}{\sigma}$ for brevity, we have 
 \begin{align*}
\sup_{\phi\in\Lip_{ }^1}\int_\T\phi d\left(f_{\omega_i*}\mu_{\omega_i}-f_{\omega_i'*}\mu_{\omega_i'}\right)&=  d_W(f_{\omega_i*}\mu_{\omega_i},f_{\omega_i'*}\mu_{\omega_i'})\\
 &\le  d_W(f_{\omega_i*}\mu_{\omega_i},f_{\omega_i'*}\mu_{\omega_i})+ d_W(f_{\omega'_i*}\mu_{\omega_i},f_{\omega_i'*}\mu_{\omega_i'}).
 \end{align*}
 
 We bound the first term from above. Notice that for any  $\xi\in\mc M_1(\T)$ and $\phi\in\Lip_{ }^1$ 
\begin{align*}
\int_{\T} \phi d(f_{\omega_i*}\xi-f_{\omega_i'*}\xi)&=\int_{\T} (\phi\circ f_{\omega_i}(x)-\phi\circ f_{\omega_i'}(x)) d\xi(x) \\
&\le \Lip(\phi) \int_{\T} |f_{\omega_i}(x)- f_{\omega_i'}(x)| d \xi(x) \\
&\le  L\,|\omega_i-\omega_i'| \le L\sigma^{-1}|\omega-\omega'|
\end{align*}
where $L$ is the Lipschitz constant of $f$. The above implies
\[
 d_W(f_{\omega_i*}\mu_{\omega_i},f_{\omega_i'*}\mu_{\omega_i})\leq L\sigma^{-1}|\omega-\omega'|.
\]
The second term can be bounded using an analogous computation
\begin{align*}
 d_W(f_{\omega_i'*}\mu_{\omega_i},f_{\omega_i'*}\mu_{\omega_i'})&\le L  d_W(\mu_{\omega_i},\mu_{\omega_i'}) \le L\Lip(\mu)\sigma^{-1}|\omega-{\omega'}| .
\end{align*}
where we used that the Lipschitz constant of $f_{\omega*}$ is equal to the  Lipschitz constant of $f_\omega$ (see Lemma \ref{Lem:LipcConstPushForw} in the Appendix) which is upper bounded by $L$ as in \eqref{As:0.3} .
 
 Putting everything together we obtain
  \[
 d_W((F_*\mu)_\omega,(F_*\mu)_{\omega'})\leq L\sigma^{-1}\left[1+\Lip(\mu)\right]\;|\omega-\omega'|.
  \]
 \end{proof}
 As a corollary to the previous proposition, for $\sigma$ sufficiently large, we obtain the existence of an invariant class of measures whose disintegration has Lipschitz dependence on the variable $\omega\in\T$, and whose Lipschitz constant goes to zero as $\sigma\rightarrow\infty$. More precisely, let's define the set $\mc M_{1,\Leb_\T}(\T\times \T)$ of probability measures on $\T\times \T$ with horizontal marginal $\Leb_\T$. Let's call $\Gamma_\ell\subset \mc M_{1,\Leb_\T}$ the set of those probability measures that have Lipschitz disintegration with Lipschitz  constant at most $\ell$:
 \[
 \Gamma_\ell:=\left\{\mu\in\mc M_{1,\Leb_\T}:\,\,\Lip(\mu)\le \ell\right\}.
 \]

 \begin{corollary}
 If  $\sigma>L$, then the set $\Gamma_{\ell}$ is invariant under the push-forward $F_*$ for every $\ell\ge \ell_0$ with 
 \[
\ell_0:=\frac{L\sigma^{-1}}{1-{L\sigma^{-1}}}.
 \] 
 \end{corollary}

  The following proposition controls the evolution of vertical marginals for  two probability measures in $\Gamma_{\ell_0}$ under application of $F_*$. In the statements below, the constants $C$ and $\lambda$ are the same as those in Assumption \eqref{As:1}.
\begin{proposition}[Approximate Memory Loss]\label{Prop:ContractionMarginalWassersteinDist}
For every $\epsilon>0$ there is $\sigma_0(\epsilon) >L$ such that if $\sigma>\sigma_0(\epsilon)$ then
\begin{itemize}
\item[i)]
\[ 
d_W(\Pi F_*^n\mu_1, \Pi F_*^n\mu_2)\le C\lambda^n+\epsilon,\quad\quad\quad \forall \mu_1,\mu_2\in\Gamma_{\ell_0};
\]
\item[ii)]
\[
d_W(\Pi F_*^n\mu, \eta_0)\le C\lambda^n+\epsilon,\quad\quad\quad \forall \mu\in\Gamma_{\ell_0};
\]
where $\eta_0$ is the stationary measure for $\mc P$.
\end{itemize}
\end{proposition}   
\begin{proof}
Let $\mu:=\mu_1-\mu_2$ and recall that $\Pi\mu=\int_\T \mu_\omega \,d\omega$ is the vertical marginal of $\mu$.  Since 
\[
d_W(\mu_\omega,\mu_{\omega'})\le \ell_0|\omega-\omega'|\le \ell_0
\]
 then $d_W(\mu_\omega, \Pi\mu)\le \ell_0$ (see Lemma \ref{Lem:AvofMeasures} in the Appendix). Therefore,
\begin{align*}
\Pi F_*\mu&=\int_{\T}f_{\omega*}\mu_\omega d\omega  =\int_{\T}f_{\omega*}(\mu_\omega-\Pi\mu)d\omega+\int_{\T}f_{\omega*}(\Pi\mu)d\omega\\
&=\int_{\T}f_{\omega*}(\mu_\omega-\Pi\mu)d\omega+\mc P(\Pi\mu),
\end{align*}
where $\mc P$ is defined in Eq. \eqref{Def:OperatorP} and, by Lemma \ref{Lem:LipcConstPushForw},  
\[
\left\|\int_{\T}f_{\omega*}(\mu_\omega-\Pi\mu)d\omega\right\|_W\leq L\ell_0.
\]

For higher iterates, one gets the telescopic sum
\begin{align}
\Pi F^n_*\mu&=\int_\T d{\omega_{n-1}}f_{\omega_{n-1}*}( (F_*^{n-1}\mu)_{\omega_{n-1}}-\Pi F_*^{n-1}\mu)+\int_\T d{\omega_{n-1}}f_{\omega_{n-1}*}( \Pi F_*^{n-1}\mu)\nonumber\\
&=\mc P^n(\Pi\mu)+\sum_{i=0}^{n-1}\int_\T d\omega_{n-1}f_{\omega_{n-1}*}...\int_\T d \omega_i f_{\omega_{i}*}((F_*^i\mu)_{\omega_i}-\Pi F_*^i\mu)\label{Eq:TelescEquations}
\end{align}
and by triangular inequality
\begin{align}\label{Eq:EstWVertMarg}
\|\Pi F^n_*\mu\|_{W}&\le \|\mc P^n\Pi\mu\|_W+ \left\|\sum_{i=0}^{n-1}\int_\T d\omega_{n-1}f_{\omega_{n-1}*}...\int_\T d \omega_i f_{\omega_{i}*}((F_*^i\mu)_{\omega_i}-\Pi F_*^i\mu)\right\|_W.
\end{align}
For the first term in the above inequality, 
\[
\|\mc P^n(\Pi \mu)\|_W=d_W(\mc P^n\Pi\mu_1,\mc P^n\Pi\mu_2)\le d_{TV}(\mc P^n\Pi\mu_1,\mc P^n\Pi\mu_2)\le C\lambda^n
\]
where  the first  inequality follows by $d_W\le d_{TV}$ (see Lemma \ref{Lem:dwlessdtv}), and  the last inequality is Assumption \ref{As:1}.
For the second term, each summand can be treated as follows
\begin{align}
&\left\|\int_\T d\omega_{n-1}f_{\omega_{n-1}*}...\int_\T d \omega_i f_{\omega_{i}*}\left [(F_*^i\mu_1)_{\omega_i}-(F_*^i\mu_2)_{\omega_i}-\Pi F_*^i\mu_1+\Pi F_*^i\mu_2\right]\right\|_W\le \nonumber\\
&\quad\quad\le \sup_\omega \Lip(f_{\omega*})^{n-1-i}\sup_\omega\left\|(F_*^i\mu_1)_{\omega_i}-(F_*^i\mu_2)_{\omega_i}-\Pi F_*^i\mu_1+\Pi F_*^i\mu_2\right\|_W\nonumber\\
&\quad\quad\leq 2 L^{n-1-i}\,\ell_0.\label{Eq:EstReminder}
\end{align}

Now, one can pick $n_0\in\N$ such that $C\lambda^{n_0}\le \epsilon/2$, and $\sigma_0>0$ so that 
\[
\frac{L\sigma_0^{-1}}{ 1 - L \sigma_0^{-1}}\sum_{i=0}^{n_0-1}L^{n_0-1-i} = \ell_0\sum_{i=0}^{n_0-1}L^{n_0-1-i}\le \epsilon/2.
\]
This way, if $n\le n_0$
\begin{align*}
\|\Pi F_*^n\mu\|_W&\le C\lambda^n+2\ell_0\sum_{i=0}^{n-1}L^i\\
&\le C\lambda^n+2\ell_0\sum_{i=0}^{n_0-1}L^i\\
&\le C\lambda^n+\epsilon/2
\end{align*}
and if $n\ge n_0$, \[d_W(\Pi F_*^n\mu_1,\Pi F_*^n\mu_2)=d_W(\Pi F_*^{n_0}F_*^{n-n_0}\mu_1,\Pi F_*^{n_0}F_*^{n-n_0}\mu_2),\] 
and since $F_*^{n-n_0}\mu_1$ and $F_*^{n-n_0}\mu_2$ both belong to $\Gamma_{\ell_0}$
\begin{align*}
d_W(\Pi F_*^{n_0}(F_*^{n-n_0}\mu_1),\Pi F_*^{n_0}(F_*^{n-n_0}\mu_2))&\le C\lambda^{n_0}+\epsilon/2 \le C\lambda^n+\epsilon
\end{align*}
which proves point i).

For point ii), going back to \eqref{Eq:TelescEquations} and picking $n_0$ and $\sigma_0$ as above,  for any $\mu\in\Gamma_{\ell_0}$ and $n\le n_0$
\begin{align*}
d_W(\Pi F_*^n\mu,\eta_0)&\le d_W(\Pi F_*^n\mu,\mc P^n\Pi\mu)+d_W(\mc P^n\Pi\mu,\eta_0)\\
&\le C\lambda^n+\epsilon/2
\end{align*}
while for $n\ge n_0$ we use {an analogous computation} and get
\begin{align*}
d_W(\Pi F_*^n\mu,\eta_0)&\le d_W(\Pi F_*^{n_0}F_*^{n-n_0}\mu,\mc P^{n_0} \Pi F_*^{n-n_0}\mu)+d_W(\mc P^{n_0} \Pi F_*^{n-n_0}\mu,\eta_0)\\
&\le C\lambda^{n_0}+\epsilon/2\\
&\le C\lambda^n+\epsilon
\end{align*}
\end{proof}

We can now proceed with the proof of the main theorem in the case without distortion.
\begin{proof}[Proof of Theorem \ref{Thm:AppDecayofCorr} under condition \eqref{As:0plus}]
Assume that $\int\psi(x)dx=0$. Then $\psi=\psi_1-\psi_2$ where {$\psi_1,\psi_2\ge0$ are the positive and negative parts of $\psi$} and $\int \psi_1=\int \psi_2=: M$. Take $\mu$ the  measure on $\T\times \T$ defined as 
\begin{equation}\label{Eq:Measuremu}
d\mu(\omega,x)=M^{-1}\left(\psi_1(x)-\psi_2(x)\right)d\omega dx.
\end{equation}
It follows that $\mu=\mu_1-\mu_2$ where $\mu_1$, $\mu_2$ are probability measures having  constant disintegrations $\mu_{1,\omega}=M^{-1}\psi_1(x)dx$ and  $\mu_{2,\omega}=M^{-1}\psi_2(x)dx$.  In particular, $\mu_1,\,\mu_2\in\Gamma_{\ell_0}$. 

Now, picking $\sigma_0$ as in Proposition \ref{Prop:ContractionMarginalWassersteinDist}, if $\sigma>\sigma_0$
\begin{align}
\left|\int_{\T\times\T}\phi\circ F^n(\omega,x) \psi(x)dxd\omega\right|&= \left|M\int_{\T\times\T}\phi(x) d(F^n_*\mu)(\omega,x)\right|\label{Eq:DecProofNoDist1}\\
&= M\left|\int_{\T}\phi(x) d(\Pi F^n_*\mu)(x)\right|\label{Eq:DecProofNoDist2}\\
&\le M \Lip(\phi)(C\lambda^n+\epsilon)\label{Eq:DecProofNoDist3}
\end{align}
where for \eqref{Eq:DecProofNoDist1} we used the duality property of the push-forward and the definition of $\mu$, and  in \eqref{Eq:DecProofNoDist3} we used  that $\phi$ does not depend on $\omega\in\T$ and Proposition \ref{Prop:ContractionMarginalWassersteinDist}.

If $\int\psi\neq 0$ consider $\tilde \psi:=\psi-\int\psi$. 
\begin{align*}
\int_{\T\times\T}\phi\circ F^n(\omega,x) \psi(x)dxd\omega&= \int_{\T\times\T}\phi\circ F^n(\omega,x) \tilde\psi(x)dxd\omega+\\
&+ \left(\int_{\T\times\T}\phi\circ F^n(\omega,x)d\omega dx-\int_\T\phi(x)d\eta_0(x)\right)\left(\int_\T\psi(x)dx\right)+\\
&+\left(\int\phi(x)d\eta_0(x)\right)\left(\int_\T\psi(x)dx\right).
\end{align*}
 
For the first term, we use \eqref{Eq:DecProofNoDist3}; for the second term
\begin{align*}
\left|\int_{\T\times\T}\phi\circ F^n(\omega,x)d\omega dx-\int_\T\phi(x)d\eta_0(x)\right|&=\left|\int_\T\phi(x)d(\Pi F_*^n\Leb_{\T\times\T}-\eta_0)(x)\right|\\
&\le \Lip(\phi)d_W(\Pi F_*^n\Leb_{\T\times\T},\eta_0)
\end{align*}
and from point ii) of Proposition \ref{Prop:ContractionMarginalWassersteinDist} the above is less than $\Lip(\phi)[C\lambda^n+\epsilon]$. By triangular inequality
\[
\left|\int_{\T\times\T}\phi\circ F^n(\omega,x) \psi(x)dxd\omega-\left(\int_\T\phi(x)d\eta_0(x)\right)\left(\int_\T\psi(x)dx\right)\right|\le C_{\varphi,\psi}(C\lambda^n+\epsilon)
\]
where $C_{\varphi,\psi}\le \frac32\|\psi\|_{L^1}(\Lip(\phi)+1)$.
\end{proof}

\begin{remark}
As a remark, note that if one tries to estimate the quantifier $\epsilon$ in Theorem \ref{Thm:AppDecayofCorr}  for a given datum,  by inspecting the proof of this simpler case, one realizes that {a bound for }$\epsilon$ is proportional to the smallest number one gets from the sequence $\left\{ \max\{C\lambda^n ,  2\sigma^{-1} L^n\}\right\}_{n \in \N} $.  Since the first sequence is decreasing, while the second is increasing, the optimal trade off is achieved when they are of about the same size. Thus imposing that the two numbers are equal,  one obtains the estimate  $\epsilon \lesssim \sigma^{-\gamma} $ for some $\gamma>0$ which depends on $C$, $\lambda$, and $L$. 
\end{remark}
\section{General case: proof of Theorem \ref{Thm:AppDecayofCorr}}\label{Sec:Distortion}

\subsection{Control on the disintegration along vertical fibres}

Take a measure $\mu_0$ on $\T^{m_1}\times \T^{m_2}$ with horizontal marginal equal to  $\nu_0\in\mc M_1(\T^{m_1})$ which is absolutely continuous with respect to Lebesgue, and let 
$\mu_1:=F_* \mu_0$.
It follows from the skew-product structure of $ F$ that the horizontal marginal of $\mu_1$ equals $\nu_1:=g_*\nu_0$. We will denote by $\rho_1$ the density of $\nu_1$\footnote{Since $g$ is a local diffeomorphism is in particular \emph{nonsingular} and its push-forward sends absolutely continuous measures to absolutely continuous measures.}.  Recall from Section \ref{sec:setting} that $g$ is a local diffeomorphism, $g_i$ are its invertible branches, and $h_i$  their
inverses. Then an explicit expression of $\rho_1$ in terms of $\rho_0$ is given by 
\[
\rho_1(\omega)=\sum_{i=1}^d\frac{\rho_0(\omega_i)}{|\D g_{\omega_i}|} \quad\forall\omega\in\T^{m_1}
\]
where we denote by $\omega_i=h_i\omega$  the preimages of $\omega$ and  $| \D g_{\omega_i} | $ for $  | \D g (\omega_i) |$.   

For $k\in\{0,1\}$, let  $\{\mu_{k,\omega}\}_{\omega\in\T^{m_1}}$ be a disintegration of $ \mu_k$  w.r.t. the measurable partition $\{\{\omega\}\times\T^{m_2}\}_{\omega\in\T^{m_1}}$. For a definition and some results on disintegrations see Appendix \ref{App:RohlinThm}.

\begin{proposition} \label{Prop:FormDisint} A disintegration $\{\mu_{1,\omega}\}_{\omega\in\T^{m_1}}$ of $\mu_1$ is given by
\begin{equation}\label{Eq:ExpMuomega}
\mu_{1,\omega}=\frac{1}{\rho_1(\omega)}\sum_{i=1}^d\frac{\rho_0(\omega_i)}{|\D g_{\omega_i}|}f_{\omega_i*}\mu_{0,\omega_i}.
 \end{equation}
 \end{proposition}
 \begin{proof}

 Let  $B_\delta(\omega)\subset \T^{m_1}$ be the Euclidean ball centered at $\omega$ of radius ${\delta}$. By Theorem \ref{Thm:DisintSimm} in  Appendix \ref{App:RohlinThm}, for $\Leb_{\T^{m_1}}-$a.e. $\omega$
\begin{equation}\label{Eq:Disint(n)}
 \mu_{1,\omega}=\lim_{{\delta}\rightarrow 0}\frac{\int_{B_\delta(\omega)}ds{\rho_1}(s){\mu_{1,s}}}{\int_{B_\delta(\omega)}ds{\rho_1}(s)}
\end{equation}
where the limit is with respect to the weak$*$ topology. Using the definition of disintegration and that ${\mu_1}(B_\delta(\omega)\times I)={\mu_0}(F^{-1}(B_\delta(\omega)\times I))$, for every measurable set $I$ on $\T^{m_2}$ one gets, for ${\delta}>0$ sufficiently small,
\[
\int_{B_\delta(\omega)}ds{\rho_1}(s){\mu_{1,s}}=\sum_{i=1}^d\int_{h_i(B_\delta(\omega))}ds {\rho_0}(s)f_{s*}{\mu_{0,s}}.
\]
By  changing  variables, $s=h_i(s')$, and multiplying and dividing by ${\rho_1}(s)$, the above equals
\begin{align*}
\int_{B_\delta(\omega)}ds'{\rho_1}(s')\left[\frac{1}{\rho_1(s')}\sum_{i=1}^d\frac{{\rho_0}(s'_i)}{|\D g_{s'_i}|}f_{s'_i*}{\mu_{0,s'_i}}\right],
\end{align*}
where we denoted $s'_i=h_i(s')$.  Applying Lebesgue's differentiation theorem, Eq. \eqref{Eq:Disint(n)} becomes
\[
{\mu_{1,\omega}}=\frac{1}{\rho_1(\omega)}\sum_{i=1}^d\frac{{\rho_0}(\omega_i)}{|\D g_{\omega_i}|}f_{\omega_i*}{\mu_{0,\omega_i}}.
\]
\end{proof}

The formula for the evolution of disintegrations in \eqref{Eq:ExpMuomega} depends on $\nu_0$ and $\nu_1$, the horizontal marginals of the measures $\mu_0$ and $\mu_1$.  Thanks to assumptions on $g$, the evolution of the horizontal can be controlled (see Lemma \ref{Lem:Cones} below).

Consider for $a\ge 0$,  the cone of log-Lipschitz functions,
\[
\mc V_a:=\left\{\phi:\T^{m_1}\rightarrow \R^+:\,\frac{\phi(\omega)}{\phi(\omega')}\le e^{a|\omega-\omega'|}\right\}.
\]
The following lemma  gathers some standard facts about uniformly expanding maps with bounded distortion, such as  $g$. 
\begin{lemma} \label{Lem:Cones} Let $g:\T^{m_1}\rightarrow \T^{m_1}$ be a $C^2$ local diffeomorphism satisfying \eqref{As:0.1}-\eqref{As:0.2}, and let $\rho_0$ and $\rho_1$ be as above.  

\
\begin{itemize}
\item[i)]  If $\rho_0\in\mc V_a$, then  $\rho_1\in\mc V_{\sigma^{-1}a+D}$. In particular, if  $a\ge a_0:= \frac{D}{1-\sigma^{-1}}$, then ${\rho_1}\in\mc V_a$;
\item[ii)] If $\rho_0\in\mc V_{a_0}$, calling $\rho_n$ the density of $\nu_n:=g_*^n\nu_0$,  there are $C_g>0$ and $\lambda_g\in(0,1)$ such that
\[
\|\rho_n-\rho_g\|_{\infty}:=\sup_{\omega\in\T^{m_1}}|\rho_n(\omega)-\rho_g(\omega)|\le C_g\lambda_g^n.
\]
\end{itemize}
\end{lemma}
\begin{proof}
See e.g. \cite{Liverani95}, \cite{SDDSViana}.
\end{proof}
\begin{remark}
In point ii) of Lemma \ref{Lem:Cones}, one can choose $C_g:=C_g(D,\sigma)$ and $\lambda_g:=\lambda_g(D,\sigma)$. Moreover, fixed $D$, the two functions can be chosen to be decreasing with respect to $\sigma$. This implies that fixed $D>0$ and $\sigma_0>1$, there are constants $\bar C$ and $\bar \lambda\in(0,1)$ such that  $C_g<\bar C$ and $\lambda_g<\bar \lambda$ for any $g$ satisfying \eqref{As:0.1}-\eqref{As:0.2} with $\sigma\ge \sigma_0$.
\end{remark}

From now on we will restrict our analysis to probability measures on $\T^{m_1}\times \T^{m_2}$ whose horizontal marginals belong to $\mc V_a$ for some $a>0$.


 \begin{proposition}\label{Prop:ControlHDist}
Assume $\rho_0\in\mc V_a$ for some  $a\ge a_0$ and that ${\mu_0}$ has Lipschitz disintegration. Then the disintegration of ${\mu_1}$ given in  \eqref{Eq:ExpMuomega} is Lipschitz  and 
 \[
\Lip \left({\mu_1}\right)\le \sigma^{-1}L \Lip\left({\mu_0}\right)+ \left[C_a+\sigma^{-1}L\right]
 \]
where $C_a:=e^{(a+\sigma^{-1}a+D)C_1}$ and $C_1$ is the diameter of $\T^{m_1}$.
 \end{proposition} 
 \begin{proof}
The proof is analogous to that of Proposition \ref{Prop:LipNodist}, although one has to work with \eqref{Eq:ExpMuomega}, rather then the simpler formula \eqref{Eq:FormDisintNodist}. 
 \begin{align*}
d_W(&\mu_{1,\omega}, \mu_{1,\omega'})=\\
 &=\sup_{\phi\in\Lip^1}\int \phi d\left(\sum_{i=1}^d\frac{1}{{\rho_1}(\omega)}\frac{{\rho_0}(\omega_i)}{|\D g_{\omega_i}|}f_{\omega_i*}{\mu_{0,\omega_i}}-\frac{1}{{\rho_1}(\omega')}\frac{{\rho_0}(\omega_i')}{|\D g_{\omega_i'}|}f_{\omega_i'*}{\mu_{0,\omega_i'}}\right)\\
 &\le\sup_{\phi\in\Lip^1}\int \phi d\left(\sum_{i=1}^d\frac{1}{{\rho_1}(\omega)}\frac{{\rho_0}(\omega_i)}{|\D g_{\omega_i}|}f_{\omega_i*}{\mu_{0,\omega_i}}-\frac{1}{{\rho_1}(\omega')}\frac{{\rho_0}(\omega_i')}{|\D g_{\omega_i'}|}f_{\omega_i*}{\mu_{0,\omega_i}}\right)+\\
 &\quad\quad +\sup_{\phi\in\Lip^1}\sum_{i=1}^d\int \phi d\left(\frac{1}{{\rho_1}(\omega')}\frac{{\rho_0}(\omega_i')}{|\D g_{\omega_i'}|}f_{\omega_i*}{\mu_{0,\omega_i}}-\frac{1}{{\rho_1}(\omega')}\frac{{\rho_0}(\omega_i')}{|\D g_{\omega_i'}|}f_{\omega_i'*}{\mu_{0,\omega_i'}}\right)\\
 &=: A+B
 \end{align*}
where to get the inequality we added and subtracted the same quantity and distributed the $\sup$.

\emph{ Upper bound for $A$.} To bound the first term
\begin{align*}
A&
=\sup_{\phi\in\Lip^1}\sum_{i=1}^d\frac{1}{\rho_1(\omega)}\frac{\rho_0(\omega_i)}{|\D g_{\omega_i}|}\left(1-\frac{\frac{1}{\rho_1(\omega')}\frac{\rho_0(\omega_i')}{|\D g_{\omega_i'}|}}{\frac{1}{\rho_1(\omega)}\frac{\rho_0(\omega_i)}{|\D g_{\omega_i}|}}\right)\int\phi d({f_{\omega_i}}_*{\mu_{0,\omega_i}})\\
&\leq \frac{1}{\rho_1(\omega)}\sum_{i=1}^d\frac{\rho_0(\omega_i)}{|\D g_{\omega_i}|}\left|1-e^{[a+\sigma^{-1}a+D]|\omega-\omega'|}\right|\\
&\le e^{[a+\sigma^{-1}a+D]C_1} |\omega-\omega'|.
\end{align*}
where $C_1>0$ is the diameter of $\T^{m_1}$. To estimate the ratio in parenthesis, we used that $\rho_0\in \mc V_a$ with $a\ge a_0$ implies ${\rho_1}\in\mc V_a$,   the fact that $|\phi|\le 1$, and \eqref{As:0.2}.

\emph{Upper bound for $B$.} The second term can be bounded by
 \begin{align*}
 d_W&\left(\sum_{i=1}^d\frac{1}{{\rho_1}(\omega')}\frac{{\rho_0}(\omega_i')}{|\D g_{\omega_i'}|}f_{\omega_i*}{\mu_{0,\omega_i}}, \sum_{i=1}^d\frac{1}{{\rho_1}(\omega')}\frac{{\rho_0}(\omega_i')}{|\D g_{\omega_i'}|}f_{\omega_i'*}{\mu_{0,\omega_i'}} \right)\leq\\
 &\quad\quad\quad\leq \max_id_W(f_{\omega_i*}{\mu_{0,\omega_i}},f_{\omega_i'*}{\mu_{0,\omega_i'}}),
 \end{align*}
 where we used that $\sum_{i=1}^d\frac{1}{{\rho_1}(\omega')}\frac{{\rho_0}(\omega_i')}{|\D g_{\omega_i'}|}=1$ and Lemma \ref{Lem:ConvexComb} about the Wasserstein distance between convex combinations of measures.
 
 By triangular inequality, 
 \[
 d_W({f_{\omega_i}}_*{\mu_{0,\omega_i}},{f_{\omega_i'}}_*{\mu_{0,\omega_i'}})\leq d_W({f_{\omega_i}}_*{\mu_{0,\omega_i}},{f_{\omega_i'}}_*{\mu_{0,\omega_i}})+ d_W({f_{\omega_i'}}_*{\mu_{0,\omega_i}},{f_{\omega_i'}}_*{\mu_{0,\omega_i'}}).
 \]
For the first term, pick any $\phi\in\Lip ^1$ and any probability measure $\xi$ 
\begin{align*}
\int \phi d(f_{\omega_i*}\xi-f_{\omega_i'*}\xi)&=\int (\phi\circ f_{\omega_i}-\phi\circ f_{\omega_i'})d\xi\\
&\le L|\omega_i-\omega_i'| \le L\sigma^{-1}|\omega-\omega'|.
\end{align*}
So
\[
d_W(f_{\omega_i*}{\mu_0}_{\omega_i},f_{\omega_i'*}{\mu_0}_{\omega_i})\leq L\sigma^{-1}|\omega-\omega'|.
\]
For the second term, using the fact that for every $\omega\in\T^{m_1}$, $\Lip(f_{\omega*})=L$,
\begin{align*}
 d_W(f_{\omega_i'*}{\mu_0}_{\omega_i},f_{\omega_i'*}{\mu_0}_{\omega_i'})&\le Ld_W({\mu_0}_{\omega_i},{\mu_0}_{\omega_i'}) \\
&\le L\Lip({\mu_0}) |\omega_i-\omega_i'|\\
&\le L\Lip({\mu_0})\sigma^{-1}\;|\omega-\omega'|,
\end{align*}
 which implies 
  \[
B\leq \sigma^{-1}L\left[1+\Lip({\mu_0}) \right]\;|\omega-\omega'|.
  \]
 
 Putting together the estimates for $A$ and $B$
 \[
 \Lip \left({\mu_1}\right)\le \sigma^{-1}L \Lip\left({\mu_0}\right)+ \left[C_a+\sigma^{-1}L\right]
 \]
 \end{proof}
 
  As a corollary to the previous proposition we obtain the existence of an invariant class of measures with  Lipschitz disintegration and with uniformly bounded Lipschitz constant. More precisely, let's define  $\Gamma_{\ell,a}$ the set of  probability measures that have a Lipschitz disintegration with constant at most $\ell$ and  horizontal marginal with density  in $\mc V_a$:
 \[
 \Gamma_{\ell,a}:=\left\{\mu\in\mc M_{1}(\T^{m_1}\times\T^{m_2}):\,\,\Lip\left(\mu \right)\le \ell\,\mbox{ and }\,\frac{d\pi_{1*}\mu}{d\Leb_{\T^{m_1}}}{\in \mc V_a}\right\}.
 \]

 \begin{corollary}\label{Cor:InvariantClasses}
 
 \
 \begin{itemize}
 \item[i)]If $\rho\in \mc V_a$ with $a\le a_0:= \frac{D}{1-\sigma^{-1}}$, $\sigma>L$, then $F_*(\Gamma_{\ell,a})\subset \Gamma_{\ell,a_0}$ for every $\ell\ge \ell_0$ with 
 \begin{equation}\label{Eq:DefEll0}
 \ell_0:=\frac{\sigma^{-1}L+C_{a_0}(1+D)}{\left(1-\sigma^{-1}L\right)}.
 \end{equation}

 \item[ii)] If there are $a>0$ and $\ell>0$ such that  $\mu\in\Gamma_{\ell,a}$, then for every ${\delta}>0$ there is $N\in\N$ such that 
 \[
 \Lip(F_*^n\mu)\le \ell_0+{\delta}
 \]
 for all $n>N$.
 \end{itemize}
 \end{corollary}
 \begin{proof}

To prove i), recall that the horizontal marginal of $F_*\mu$ is the push-forward under $g$ of the horizontal marginal of $\mu$. Since the horizontal marginal of $\mu$ has density in $\mc V_{a_0}$, by the inclusion $\mc V_a\subset \mc V_{a_0}$ and  Lemma \ref{Lem:Cones}, the horizontal  marginal of $F_*\mu$ belongs to $\mc V_{a_0}$. By Proposition \ref{Prop:ControlHDist} and the choice of $\ell_0$, it follows that if $\Lip(\mu)\le \ell_0$ then also $\Lip(F_*\mu)\le \ell_0$.
 
 For point ii) notice that  the horizontal marginal of $F_*^n\mu$ belongs to $\Gamma_{\ell_n,a_n}$ for some $a_n$ and $\ell_n$  such that  $a_n\rightarrow a_0$ as $n\rightarrow \infty$ by  Lemma \ref{Lem:Cones},   and $\ell_n\rightarrow \ell_0$ as $n\rightarrow \infty$ by  Proposition \ref{Prop:ControlHDist}. The claim follows easily.
 \end{proof}
 
\subsection{Tracking the evolution of the vertical marginal.}\label{Sec:TrackingDistVertMarg}
Let's consider $\mc M_{1,\nu_g}(\T^{m_1}\times \T^{m_2})$ the set of Borel probability measures on $\T^{m_1}\times \T^{m_2}$ having horizontal marginal equal to $\nu_g$,  the invariant measure for $g$, and recall that for $\mu \in \mc M_{1,\nu_g}(\T^{m_1}\times \T^{m_2})$, the vertical marginal is given by 
\[
 \Pi \mu = \int_{\T^{m_2}} d\omega \rho_g (\omega) \mu_\omega.
\]

For every $i=1,...,d$, call
\[
{\bar\rho_i}:=\nu_g(I_i)=\int_{I_i}d\omega\rho_g(\omega)
\]
the measure of $I_i$ with respect to the invariant measure of $g$.  Define the map ${\bo \Delta}:\mc M_{1,\nu_g}(\T^{m_1}\times\T^{m_2})\rightarrow (\mc M_1(\T^{m_2}))^d$ in the following way
\[
({\bo \Delta} \mu)_i:= \bar{\rho}_i^{-1}\int_{I_i}d\omega\rho_g(\omega)\mu_\omega,
\]
i.e. $({\bo \Delta} \mu)_i$ is the average of the disintegration $\{\mu_\omega\}_{\omega\in\T^{m_1}}$ on $I_i$ with respect to the invariant measure of $g$. The map $\bo\Delta$ gives a decomposition of $\mu$ which can be viewed as a \emph{coarse-graining} of the disintegration of $\{\mu_\omega\}_{\omega\in\T^{m_1}}$. Moreover, 
for any $\mu\in \mc M_{1,\nu_g}(\T^{m_1}\times\T^{m_2})$
\[
\Pi\mu=\sum_{i=1}^d{\bar\rho_i}({\bo \Delta} \mu)_i,
\]
therefore, by  keeping track of ${\bo \Delta}(F_*^n\mu)$, we can keep track of $\Pi F_*^n\mu$. 

Consider also
\begin{equation}
\mc F_i:={\bar\rho_i}^{-1}\int_{I_i} d\omega\rho_g(\omega) f_{\omega*},
\end{equation}
which is the average of the operators $\{f_{\omega*}\}_{\omega\in \T^{m_1}}$ on $I_i$ w.r.t. $\nu_g$ restricted to $I_i$ and normalized. A lemma below shows that $\mc F_i$ is an approximation of $f_{\omega*}$ for $\omega\in I_i$. The smaller is the size of $I_i$, i.e. the larger is $\sigma>0$, the better is the approximation.

For every $1\le i,j \le d$, define the operators
\begin{equation}\label{Eq:fi}
\mc L_{ij}:={\bar\rho_i}^{-1}\left(\int_{I_i}d\omega\frac{\rho_g(\omega_j)}{|\D g_{\omega_j}|}\right)\mc F_{j }; 
\end{equation}
 and consider the operator ${\bo {\mc L}}:(\mc M_1(\T^{m_2}))^d\rightarrow (\mc M_1(\T^{m_2}))^d$
\begin{equation}\label{Eq:OpL}
({\bo {\mc L}}\bo\mu)_i=\sum_{j=1}^d \mc L_{ij}(\bo \mu)_j.
\end{equation}

\begin{remark}
Before moving on, let us stress why the above mappings ${\bo \Delta}$ and ${\bo {\mc L}}$ are important: ${\bo \Delta}(F_*\mu)$ and  ${\bo {\mc L}}({\bo \Delta} \mu)$ are very close when the expansion of $g$ is very large. This will let us prove that for fixed $n$, we can approximate  ${\bo \Delta}(F^n_*\mu)$ with ${\bo {\mc L}}^n({\bo \Delta} \mu)$ when the expansion of $g$ is sufficiently large, with the advantage that $\bo{\mc L}$ has good contraction properties.
\end{remark}

The remark above is formalised in the following propositions. For $\bo \mu_1,\bo\mu_2\in(\mc M_1(\T^{m_2}))^d$ we define
\[
d_W(\bo\mu_1,\bo\mu_2)=\max_{i=1, \ldots, d}d_W((\bo \mu_{1})_i,(\bo\mu_{2})_i).
\]
\begin{proposition}\label{Prop:ApproximDisitnMathcalL}
If $\mu\in  \Gamma_{\ell_0,a_0}$, with $\ell_0$ and $a_0$ as in Corollary \ref{Cor:InvariantClasses}, then there is a constant $K_\#>0$ uniform in $\sigma$, and $C_3:(1,+\infty)\rightarrow \R^+$ decreasing such that 
\begin{equation}\label{Eq:ApproximDisitnMathcalL}
d_W({\bo \Delta}(F_*^n\mu), {\bo {\mc L}}^n({\bo \Delta}\mu))< K_\# L^{n+1}\left(\,\sigma^{-1}+C_3(\sigma)\|\rho_0-\rho_g\|_{\infty}\,\right) 
\end{equation}
where $\rho_0$ is the density of the horizontal marginal of $\mu$.
\end{proposition}
\begin{proof}
Let's call $\nu_0$ the horizontal marginal of $\mu$, and $\rho_0\in\mc V_{a_0}$ its density.  Let's denote by $\nu_n:=F_*^n\nu_0$ and by $\rho_n:=\frac{d\nu_n}{d\Leb}$.  By Lemma \ref{Lem:Cones},  $\rho_n\in\mc V_{a_0}$ for every $n\in\N_0$.

First, let's prove \eqref{Eq:ApproximDisitnMathcalL} for $n=1$. Recalling the disintegration \eqref{Eq:ExpMuomega},
\begin{align}
({\bo \Delta}(F_*\mu))_i&={\bar\rho_i}^{-1}\int_{I_i}d\omega \rho_g(\omega)(F_*\mu)_{\omega}\nonumber\\
&={\bar\rho_i}^{-1}\int_{I_i}d\omega \frac{\rho_g(\omega)}{\rho_1(\omega)}\sum_{j=1}^d\frac{\rho_0(\omega_j)}{|\D g_{\omega_j}|}f_{\omega_j*}\mu^{}_{\omega_j}\nonumber\\
&={\bar\rho_i}^{-1}\int_{I_i} d\omega\left(\frac{\rho_g(\omega)}{\rho_1(\omega)}-1\right)\sum_{j=1}^d\frac{\rho_0(\omega_j)}{|\D g_{\omega_j}|}f_{\omega_j*}\mu^{}_{\omega_j}+\label{Est:mathcalLHm1}\\
&\quad\quad+{\bar\rho_i}^{-1}\int_{I_i}d\omega\sum_{j=1}^d\frac{\rho_0(\omega_j)-\rho_g(\omega_j)}{|\D g_{\omega_j}|}f_{\omega_j*}\mu_{\omega_j}+\label{Est:mathcalLHm2}	\\
&\quad\quad+{\bar\rho_i}^{-1}\int_{I_i}d\omega \sum_{j=1}^d\frac{\rho_g(\omega_j)}{|\D g_{\omega_j}|}f_{\omega_j*}\mu_{\omega_j}\label{Est:mathcalLHm3}
\end{align}
where in the last equality we added and subtracted the same terms.  We denote by $A$ the term in \eqref{Est:mathcalLHm1} and by $B$ the term in \eqref{Est:mathcalLHm2}. With this notation,
\begin{align*}
d_W(({\bo \Delta}(F_*\mu))_i, ({\bo {\mc L}}{\bo \Delta}\mu)_i)&= \left \| A+B+ \sum_{j=1}^d{\bar\rho_i}^{-1}\int_{I_i}d\omega\frac{\rho_g(\omega_j)}{|\D g_{\omega_j}|}(f_{\omega_j*}-\mc F_j)\mu_{\omega_j}\right.+\\
&\quad\quad\quad+\left.\sum_{j=1}^d{\bar\rho_i}^{-1}\int_{I_i}d\omega\frac{\rho_g(\omega_j)}{|\D g_{\omega_j}|}\mc F_j(\mu_{\omega_j}-({\bo \Delta} \mu)_j)\right\|_W.
\end{align*}
Call $\Lip_0^1(\T^{m_2};\R)$, $\Lip_0^1$ for brevity, the set of Lipschitz functions  from $\T^{m_2}$ to $\R$ with zero integral. When compute the above $\|\cdot\|_W$, taking the supremum over $\Lip^1$ or $\Lip^1_0$ doesn't matter, as the integrals of $\phi$ and that of $\phi-\int\phi$ are the same. Notice that for $\phi\in\Lip_0^1$, $|\phi|\le C_2$, where $C_2$ is the diameter of $\T^{m_2}$.
\begin{align}
&d_W(({\bo \Delta}(F_*\mu))_i, ({\bo {\mc L}}{\bo \Delta}\mu)_i)=\nonumber\\
&\quad\quad\sup_{\phi\in\Lip^1_0}\int \phi d\left(  A+B+\sum_{j=1}^d{\bar\rho_i}^{-1}\int_{I_i}d\omega\frac{\rho_g(\omega_j)}{|\D g_{\omega_j}|}(f_{\omega_j*}-\mc F_j)\mu_{\omega_j}\right.\label{Eq:EstdWFL}\\
&\quad+\left.\sum_{j=1}^d{\bar\rho_i}^{-1}\int_{I_i}d\omega\frac{\rho_g(\omega_j)}{|\D g_{\omega_j}|}\mc F_j(\mu_{\omega_j}-({\bo \Delta} \mu)_j)\right).\label{Eq:EstdWFL2}
\end{align}

Let's call 
\[
\delta_n:=\|\rho_n-\rho_g\|_{\infty}=\sup_{\omega\in\T^{m_1}}|\rho_n(\omega)-\rho_g(\omega)|.
\]
Since $\rho_1\in\mc V_a$, $|\rho_1|\ge e^{-C_1D}$ where $C_1>0$ is the diameter of $\T^{m_1}$. Therefore
\begin{align*}
 \left|1-\frac{\rho_g(\omega)}{\rho_1(\omega)}\right|\le \frac{1}{\rho_1(\omega)}|\rho_g(\omega)-\rho_1(\omega)|\le e^{C_1D}\delta_1.
\end{align*}

Now we distribute the $\sup$ among the four terms on the RHS of \eqref{Eq:EstdWFL}, and estimate each of them separately.
\begin{align*}
\sup_{\phi\in\Lip^1_0}\int \phi dA&\le {\bar\rho_i}^{-1}\int_{I_i} d\omega\sum_{j=1}^d\frac{\rho_0(\omega_j)}{|\D g_{\omega_j}|}\,\,\sup_{\phi\in\Lip_0^1}\left|\int_{\T^{m_2}}\phi(x)d\left[\left(\frac{\rho_g(\omega)}{\rho_1(\omega)}-1\right)f_{\omega_j*}\mu^{}_{\omega_j}\right](x)\right|\\
&\le  {\bar\rho_i}^{-1}\int_{I_i} d\omega\sum_{j=1}^d\frac{\rho_0(\omega_j)}{|\D g_{\omega_j}|}C_2e^{C_1D}\delta_1\\
&=C_2e^{C_1D}\delta_1{\bar\rho_i}^{-1}\int_{I_i}d\omega\rho_1(\omega)\\
&=C_2e^{C_1D}\delta_1\frac{\nu_1(I_i)}{\nu_g(I_i)}\\
&\le C_2e^{3C_1D}\delta_1. 
\end{align*}
Then for $B$
\begin{align*}
\sup_{\phi\in\Lip^1_0}\int \phi dB&=\sup_{\phi\in\Lip^1_0} \int_{\T^{m_2}}\phi(x)d\left[{\bar\rho_i}^{-1}\int_{I_i}\sum_{j=1}^d\frac{\rho_0(\omega_j)-\rho_g(\omega_j)}{|\D g_{\omega_j}|}f_{\omega_j*}\mu_{\omega_j}d\omega\right](x)\\
&\le {\bar\rho_i}^{-1}\int_{I_i}d\omega\sum_{j=1}^d\frac{1}{|\D g_{\omega_j}|}\sup_{\phi\in \Lip^1_0}\left|\int \phi(x)(\rho_0(\omega_j)-\rho_g(\omega_j)) df_{\omega_j*}\mu_{\omega_j}(x)\right|\\
&\le {\bar\rho_i}^{-1}\int_{I_i}d\omega\sum_{j=1}^d\frac{1}{|\D g_{\omega_j}|}C_2\delta_0\\
&=\frac{g_*\Leb_{\T^{m_1}}(I_i)}{\nu_g(I_i)}C_2\delta_0\\
&\le e^{2C_1D}C_2\delta_0
\end{align*}

For the third term in the big parenthesis of Eq. \eqref{Eq:EstdWFL}, using the definition of $\mc F_j$
\begin{align*}
&\left|\int \phi(x)\sum_{j=1}^d{\bar\rho_i}^{-1}\int_{I_i}d\omega \frac{\rho_g(\omega_j)}{|\D g_{\omega_j}|}d(f_{\omega_j*}-\mc F_j)\mu_{\omega_j}(x)\right|=\\
&=\left| \sum_{j=1}^d\int_{I_j}{\bar\rho_j}^{-1}d\omega'{\bar\rho_i}^{-1}\int_{I_i}d\omega\frac{\rho_g(\omega_j)}{|\D g_{\omega_j}|}\int \phi(x)d(f_{\omega_j*}-f_{\omega'*})\mu_{\omega_j}(x)\right|\\
&=\left| \sum_{j=1}^d\int_{I_j}{\bar\rho_j}^{-1}d\omega'{\bar\rho_i}^{-1}\int_{I_i}d\omega\frac{\rho_g(\omega_j)}{|\D g_{\omega_j}|}\int(\phi\circ f_{\omega_j}(x)-\phi\circ f_{\omega'}(x))d\mu_{\omega_j}(x)\right|\\
&\le \sum_{j=1}^d\int_{I_j}{\bar\rho_j}^{-1}d\omega'{\bar\rho_i}^{-1}\int_{I_i}d\omega\frac{\rho_g(\omega_j)}{|\D g_{\omega_j}|}\int|\phi\circ f_{\omega_j}(x)-\phi\circ f_{\omega'}(x)|d\mu_{\omega_j}(x)\\
&\le  \sum_{j=1}^d\int_{I_j}{\bar\rho_j}^{-1}d\omega'{\bar\rho_i}^{-1}\int_{I_i}d\omega\frac{\rho_g(\omega_j)}{|\D g_{\omega_j}|} L\diam (I_j)\\
&\le L\sigma^{-1} C_1\sum_{j=1}^d{\bar\rho_j}^{-1}|I_j|{\bar\rho_i}^{-1}\int_{I_i}d\omega\frac{\rho_g(\omega_j)}{|\D g_{\omega_j}|}\\
&\le  L\sigma^{-1} C_1\sum_{j=1}^d{\bar\rho_i}^{-1}\int_{I_i}d\omega\frac{\rho_g(\omega_j)}{|\D g_{\omega_j}|}\\
&\le   L\sigma^{-1} C_1
\end{align*}
where we used that
\[
|\phi\circ f_{\omega_j}(x)-\phi\circ f_{\omega'}(x)|\le |f_{\omega_j}(x)- f_{\omega'}(x)|\le L|\omega_j-\omega'|\le L\diam (I_j)\le L\sigma^{-1}C_1,
\]
 recall that $L$ is the Lipschitz constant of $f$ and $C_1=\diam(\T^{m_1})$; and that 
 \begin{equation}\label{Eq:resumform}
 {\bar\rho_i}^{-1}\int_{I_i}d\omega\sum_{j=1}^d\frac{\rho_g(\omega_j)}{|\D g_{\omega_j}|}={\bar\rho_i}^{-1}\int_{I_i}d\omega\rho_g(\omega)=1.
 \end{equation}

For the last term, on \eqref{Eq:EstdWFL2},
\begin{align}
&\int \phi(x) d\left( \sum_{j=1}^d{\bar\rho_i}^{-1}\int_{I_i}d\omega\frac{\rho_g(\omega_j)}{|\D g_{\omega_j}|}\mc F_j(\mu_{\omega_j}-({\bo \Delta} \mu)_j)\right)(x)=\nonumber\\
&\sum_{j=1}^d{\bar\rho_i}^{-1}\int_{I_i}d\omega\frac{\rho_g(\omega_j)}{|\D g_{\omega_j}|}\int \phi(x)d\mc F_j(\mu_{\omega_j}-({\bo \Delta} \mu)_j)(x)\nonumber\\
&\le \sum_{j=1}^d{\bar\rho_i}^{-1}\int_{I_i}d\omega\frac{\rho_g(\omega_j)}{|\D g_{\omega_j}|}\Lip(\mc F_j)d_W(\mu_{\omega_j},({\bo \Delta} \mu)_j)\\
&\le L\ell_0C_1\sigma^{-1}
\end{align}
where in the last step we used that $\Lip(\mc F_j)\le \sup_\omega\Lip(f_{\omega*})\le L$ and that 
\[
d_W(\mu_{\omega_j},({\bo \Delta} \mu)_j)\le \ell_0|\diam (I_j)\le \ell_0C_1\sigma^{-1}.
\]

Putting all of the above together we conclude that there is $K_\#>0$ (independent of $\sigma>0,$) such that 
 \[
 d_W(({\bo \Delta}(F_*\mu))_i, ({\bo {\mc L}}{\bo \Delta}\mu)_i)\le K_\#(\sigma^{-1}+\delta_0+\delta_1).
 \]
 Now, since for every $k\in\N$, $F_*^k\mu\in\mc M_{1,\nu_k}(\T^{m_1}\times\T^{m_1})\cap\Gamma_{\ell_0}$, by repeated applications of the triangular inequality 
\begin{align*}
d_W({\bo \Delta} F_*^n\mu,{\bo {\mc L}}^n{\bo \Delta}\mu)&\le \sum_{k=0}^{n-1} d_W({\bo {\mc L}}^{n-k-1}{\bo \Delta} F_*(F_*^{k}\mu),{\bo {\mc L}}^{n-k}{\bo \Delta}(F_*^{k}\mu))\\
& \le \sum_{k=0}^{n-1}\left (\sup_{\omega}\Lip(f_{\omega*})\right)^{n-k-1}d_W({\bo \Delta} F_*(F_*^{k}\mu),{\bo {\mc L}}{\bo \Delta}(F_*^{k}\mu))\\
&\le  K_\#\sum_{k=0}^{n-1} L^{n-k}\left[\ell_0\sigma^{-1}+\delta_k+\delta_{k+1}\right]\\
&\le K_\#\left[\ell_0\sigma^{-1}+\sum_{k=0}^{n-1}\delta_k+\delta_{k+1}\right]\sum_{k=0}^{n-1} L^{n-k}\\
&\le K_\#(\ell_0\sigma^{-1}+C_3\delta_0)\sum_{k=0}^{n-1} L^{n-k}.
\end{align*}
where we used that by Lemma \ref{Lem:Cones},  $\delta_k\le C_g\lambda_g^k\delta_0$, and $C_3:=\frac{2C_g}{1-\lambda_g}$.

\end{proof}

The operator ${\bo {\mc L}}$ has good spectral properties. To prove it, we are going to need the following lemma
\begin{lemma}\label{Lem:pEstimate}
The following inequality holds
\[
{\bar\rho_i}^{-1}{\bar\rho_j}^{-1}\int_{I_i}d\omega \frac{\rho_g(\omega_j)}{|\D g_{\omega_j}|}>p. 
\]
with 
\[
p:=e^{-C_1D[\frac{3}{1-\sigma^{-1}}+1]}\in(0,1)
\]
where $D$ is the bound on the distortion of the map $g$, and $C_1$ is the diameter of $\T^{m_1}$.
\end{lemma}
\begin{proof}
Recall that $\omega_j$ is shorthand notation for $h_j(\omega)$. Since
\[
\int_{\T^{m_1}} d\omega\frac{1}{|\D g_{\omega_j}|}=\Leb_{\T^{m_1}}(I_j)=:|I_j|
\]
and $|\D g\circ h_j|$ is continuous, there is $\omega_0$ such that
\[
\frac{1}{|\D g(h_j(\omega_0))|}=|I_j|
\]
Recalling the notation $\omega_j=h_j(\omega)$, the bound on the distortion \eqref{As:0.2} gives
\[
|I_j|^{-1}|\D g_{\omega_j}|^{-1}=\frac{|\D g(h_j(\omega_0))|}{|\D g(h_j(\omega))|}\ge e^{-D|\omega-\omega_0|}\ge e^{-DC_1}
\]
where $C_1$ equals the diameter of $\T^{m_1}$ w.r.t. the Euclidean distance. 

Also, recall that $\rho_g\in\mc V_{a_0}$ with $a_0=\frac{D}{1-\sigma^{-1}}$, therefore $e^{-C_1a_0}\le \rho_g\le e^{C_1a_0}$ and ${\bar\rho_i}\le e^{C_1a_0}|I_i|$. 
Putting the above considerations together 
\begin{align*}
{\bar\rho_i}^{-1}{\bar\rho_j}^{-1}\int_{I_i}d\omega \frac{\rho_g(\omega_j)}{|\D g_{\omega_j}|}&\ge e^{-C_1a_0}|I_j|^{-1}e^{-C_1a_0}|I_i|^{-1}\int_{I_i}d\omega\frac{e^{-C_1a_0}}{|\D g_{\omega_j}|}\\
&\ge e^{-3C_1a_0}|I_i|^{-1}\int_{I_i}e^{-DC_1}\\
&\ge e^{-3C_1a_0-DC_1}
\end{align*} 
\end{proof}
\begin{remark}
Notice that $p$ depends on $\sigma$, but for $D$ fixed, $p$ increases with $\sigma>1$. In particular, assuming that $\sigma\ge \sigma_0>1$, we get
\[
p\ge e^{-C_1D\left[1+\frac{3}{1-\sigma_0^{-1}}\right]}.
\]
\end{remark}

In a proposition below we show that the operator $\bo{\mc L}$ has good contracting properties with respect to $d_{TV}$.  First we state a couple of lemmas and definitions.
\begin{lemma}\label{Lem:ContmcLinTV}
For every $\bo\mu_1,\bo\mu_2\in(\mc M_1(\T))^d$
\begin{equation}
d_{TV}(\bo {\mc L}\bo\mu_1,\bo{\mc L}\bo\mu_2)\le d_{TV}(\bo\mu_1,\bo\mu_2).
\end{equation}
\end{lemma}
\begin{proof}
By definition of Total Variation distance, transfer operators are weak contractions with respect to $d_{TV}$; in particular,
for any $\eta_1,\eta_2\in\mc M_1(\T)$ and any $\omega\in\T^{m_1}$
\[
d_{TV}(f_{\omega*}\eta_1,f_{\omega*}\eta_2)\le d_{TV}(\eta_1,\eta_2),
\]
and therefore 
\begin{equation}\label{Eq:ContmcFj}
d_{TV}(\mc F_j\eta_1,\mc F_j\eta_2)\le d_{TV}(\eta_1,\eta_2)
\end{equation}
for any $j$. 

{By formula \ref{Eq:resumform}} one gets that for every $i$
\begin{align*}
d_{TV}&((\bo {\mc L}\bo\mu_1)_i,(\bo{\mc L}\bo\mu_2)_i)=\\
&\quad=d_{TV}\left(\sum_j{\bar\rho_i}^{-1}\left(\int_{I_i}d\omega\frac{\rho_g(\omega_j)}{|\D g_{\omega_j}|}\right)\mc F_j(\bo\mu_1)_j,\sum_j{\bar\rho_i}^{-1}\left(\int_{I_i}d\omega\frac{\rho_g(\omega_j)}{|\D g_{\omega_j}|}\right)\mc F_j(\bo\mu_2)_j\right)\\
&\quad\le \sum_j{\bar\rho_i}^{-1}\left(\int_{I_i}d\omega\frac{\rho_g(\omega_j)}{|\D g_{\omega_j}|}\right) d_{TV}(\mc F_j(\bo\mu_1)_j,\mc F_j(\bo\mu_2)_j)\\
&\quad\le  \sum_j{\bar\rho_i}^{-1}\left(\int_{I_i}d\omega\frac{\rho_g(\omega_j)}{|\D g_{\omega_j}|}\right) d_{TV}((\bo\mu_1)_j,(\bo\mu_2)_j)\\
&\quad\le d_{TV}(\bo\mu_1,\bo\mu_2).
\end{align*}
\end{proof}

Lemma \ref{Lem:pEstimate}  implies that ${\bo {\mc L}}$ can be decomposed in the following way: there are $p_{ij}\ge 0$ such that 
\begin{equation}\label{Eq:Takeoutp}
(\bo{\mc L})_{ij}=\mc L_{ij}=p\int_{I_j}d\omega\rho_g(\omega)f_{\omega*}+p_{ij}\mc F_j. 
\end{equation}
Define ${\bo {\mc L}}_1:(\mc M_1(\T^{m_2}))^d\rightarrow (\mc M_1(\T^{m_2}))^d$ as 
\[
({\bo {\mc L}}_1)_{ij}=\int_{I_j}d\omega \rho_g(\omega)f_{\omega *}
\] 
and ${\bo {\mc L}}_2:(\mc M_1(\T^{m_2}))^d\rightarrow (\mc M_1(\T^{m_2}))^d$ as
\[
({\bo {\mc L}}_2)_{ij}:=p_{ij}\mc F_j
\]
so that $\mc L=p\mc L_1+\mc L_2$.

\begin{lemma}
For every $n\in\N$ the following decomposition holds 
\[
{\bo {\mc L}}^n\bo\mu:=p^n{\bo {\mc L}}_1^n\bo\mu+(1-p^n)\bo R_n\bo\mu.
\]
where $\bo R_n:(\mc M_1(\T^{m_2}))^d\rightarrow (\mc M_1(\T^{m_2}))^d$ is such that 
\[
d_{TV}(\bo R_n\bo\mu_1,\bo R_n\bo\mu_2)\le d_{TV}(\bo\mu_1,\bo\mu_2)
\] 
for all $\bo\mu_1,\bo\mu_2\in(\mc M_1(\T^{m_2}))^{d}$.
\end{lemma}
\begin{proof}
Let's start  noticing that, by definition, $\sum_j p_{ij}=(1-p)$ for every $i$, in fact comparing equations \eqref{Eq:Takeoutp} and \eqref{Eq:fi} follows that 
\[
p{\bar\rho_i}+p_{ij}={\bar\rho_j}^{-1}\left(\int_{I_i}d\omega\frac{\rho_g(\omega_j)}{|\D g_{\omega_j}|}\right)
\]
and 
\[
\sum_j(p{\bar\rho_j}+p_{ij})=p+\sum_jp_{ij}=\sum_j{\bar\rho_i}^{-1}\left(\int_{I_i}d\omega\frac{\rho_g(\omega_j)}{|\D g_{\omega_j}|}\right)=1.
\]

Now we prove the statement of the lemma by induction on $n\in\N$. For $n=1$, $\bo R_1=(1-p)^{-1}\bo{\mc L_2}$ and recalling \eqref{Eq:ContmcFj}
\begin{align*}
d_{TV}((1-p)^{-1}\bo{\mc L_2} & \bo\mu_1,(1-p)^{-1}\bo{\mc L_2}\bo\mu_2)=\\
&=\max_id_{TV}\left(\sum_j(1-p)^{-1}p_{ij}\mc F_j(\bo\mu_1)_j,\sum_j(1-p)^{-1}p_{ij}\mc F_j(\bo\mu_2)_j\right)\\
&\le \max_i \sum_j(1-p)^{-1}p_{ij} d_{TV}((\bo\mu_1)_j,(\bo\mu_2)_j)\\
&\le d_{TV}(\bo\mu_1,\bo\mu_2).
\end{align*}

Now assume that the statement is true for $n-1$. 
\begin{align*}
\bo{\mc L}^{n}=\bo{\mc L}\bo{\mc L}^{n-1}&=p^{n}{\bo{\mc L}_1}^{n}+p^{n-1}(1-p){\bo R}_1{\bo{\mc L}_1}^{n-1}+(1-p^{n-1})\bo{\mc L}\bo R_{n-1}.
\end{align*}
Define
\[
\bo R_n:=\frac{(1-p)p^{n-1}{\bo R}_1{\bo{\mc L}_1}^{n-1}+(1-p^{n-1})\bo{\mc L}\bo R_{n-1}}{1-p^n}
\]
and by Lemma \ref{Lem:ContmcLinTV} applied to $\bo{\mc L}$ and $\bo{{\mc L}_1}$, the inductive step,  $(1-p)p^{n-1}+(1-p^{n-1})=1-p^n$, and Lemma \ref{Lem:ConvexComb}
\begin{align*}
d_{TV}(\bo R_n\bo\mu_1,\bo R_n\bo\mu_2)\le d_{TV}(\bo\mu_1,\bo\mu_2).
\end{align*}
\end{proof}

We are now ready to show that $\bo{\mc L}$ has good contraction properties with respect to the Total Variation distance. The proof uses a coupling argument.
\begin{proposition}\label{Prop:SpecPropMCL}
There are $C_{{\bo {\mc L}}}>0$ and $\lambda_{\bo {\mc L}}\in(0,1)$ such that  for any   $\bo \mu_1,\bo \mu_2\in (\mc M_1(\T^{m_2}))^d$
\begin{equation}\label{Eq:ContPropMCL}
d_{TV}({\bo {\mc L}}^n\bo\mu_1, {\bo {\mc L}}^n\bo\mu_2)\le C_{{\bo {\mc L}}}\lambda_{{\bo {\mc L}}}^n d_{TV}(\bo\mu_1,\bo\mu_2).
\end{equation}
\end{proposition}
\begin{proof}
Notice that all the rows of the operator ${\bo{\mc L}_1}$ are equal, therefore, for any $\bo\mu\in(\mc M_1(\T^{m_2}))^d$, also all the components of $\bo{\mc L}_1\bo\mu$ are equal, i.e. there is $\mu'\in\mc M_1(\T^{m_2})$ such that $({\bo {\mc L}}_1\bo\mu)_i=\mu'$. By definition of  ${\bo{\mc L}_1}$  follows that
\[
(\bo{\mc L}_1^2\bo\mu)_i=\sum_{j=1}^d({\bo{\mc L}_1})_{ij}\mu'=\sum_{j=1}^d\int_{I_j}d\omega\rho_j(\omega)f_{\omega*}\mu'=\mc P\mu'
\]
and by induction
\begin{equation}\label{Eq:ExpL1n}
({\bo {\mc L}}_1^n\bo\mu)_i=\mc P^{n-1}\mu'
\end{equation}
for every $i$ and $n>1$.

Pick $n_0>1$ such that $C\lambda^{n_0-1}\le \frac 12$. Then it follows from \eqref{Eq:ExpL1n} and Assumption \eqref{As:1} that for all $n>1$
\[
d_{TV}({\bo {\mc L}}_1^n\bo\mu_1,{\bo {\mc L}}_1^n\bo\mu_2)\le \frac12d_{TV}(\bo\mu_1,\bo\mu_2).
\]
which implies
\begin{align*}
d_{TV}({\bo {\mc L}}^{n_0}\bo\mu_1,{\bo {\mc L}}^{n_0}\bo\mu_2)&\le p^{n_0}d_{TV}({\bo {\mc L}}_1^{n_0}\bo\mu_1,{\bo {\mc L}}_1^{n_0}\bo\mu_2)+(1-p^{n_0})d_{TV}(R_{n_0}\bo\mu_1,R_{n_0}\bo\mu_2)\\
&\le \left(1-\frac12p^{n_0}\right)d_{TV}(\bo\mu_1,\bo\mu_2).
\end{align*}

Define $\lambda_{\bo{\mc L}}:=\left(1-\frac12p^{n_0}\right)^\frac{1}{n_0}$ and $C_{\bo{\mc L}}:=\lambda_{\bo{\mc L}}^{-n_0}$. For every $n\in\N$ there are $k\in\N$ and $0\le r<n_0$ such that $n=kn_0+r$, and by Lemma \ref{Lem:ContmcLinTV}
\begin{align*}
d_{TV}(\mc L^n\bo\mu_1,\mc L^n\bo\mu_2) &= d_{TV}(\mc L^r\mc L^{kn_0}\bo\mu_1, \mc L^r\mc L^{kn_0}\bo\mu_2)\\
&\le d_{TV}(\mc L^{kn_0}\bo\mu_1,\mc L^{kn_0}\bo\mu_2)\\
&\le \lambda_{\bo{\mc L}}^{n_0k}d_{TV}(\bo\mu_1,\bo\mu_2)\\
&\le   C_{\bo{\mc L}}\lambda_{\bo{\mc L}}^nd_{TV}(\bo\mu_1,\bo\mu_2).
\end{align*}
 \end{proof}
 The contraction properties of $\bo{\mc L}$, \eqref{Eq:ContPropMCL}, and the weak*-compactness of $(\mc M_1(\T^{m_2}))^d$ imply the existence of $\bo \eta_0\in (\mc M_1(\T^{m_2}))^d$ such that 
  \[
  \bo{\mc L}\bo\eta_0=\bo\eta_0.
  \]
  
 The following proposition is the analogous of Proposition \ref{Prop:ContractionMarginalWassersteinDist} in the case without distortion and proves approximated memory loss for the vertical marginals under application of $F_*$.  In this case, there is an extra difficulty as, in order to prove Theorem \ref{Thm:AppDecayofCorr}, the class of probability measures we start from should include those having horizontal marginal equal to $\Leb_{\T^{m_1}}$  which in general can be different from the invariant measure $\nu_g$. 
\begin{proposition}[Approximate Memory Loss]\label{Prop:KantDistanceDistCase}
Fix $D,L>0$. Given any $\epsilon>0$, there is $C_{\bo{\mc L}}''>0$ and $\sigma_0>L$ such that for any $\sigma>\sigma_0$, $F$ satisfying \eqref{As:0.1}-\eqref{As:0.3} 
\begin{itemize}
\item[i)]
\[
d_W(\Pi F_*^t\mu_1,\Pi F_*^t\mu_2) \le C''_{{\bo {\mc L}}}\lambda_{{\bo {\mc L}}}^{t}+\epsilon,\quad\quad\forall t\in\N
\]
for any  $\mu_1,\mu_2\in  \Gamma_{\ell_0,a_0}$, $\ell_0$ defined in \eqref{Eq:DefEll0};
\item[ii)]
\[
d_W\left(\Pi F_*^t\mu_,\sum_{i=1}^d{\bar\rho_i}(\bo \eta_0)_i\right)\le  C''_{{\bo {\mc L}}}\lambda_{{\bo {\mc L}}}^{t}+\epsilon,\quad\quad\forall t\in\N
\]
for any  $\mu\in\Gamma_{a_0,\ell_0}$, $\ell_0$ defined in \eqref{Eq:DefEll0}.
\end{itemize}
\end{proposition}
\begin{proof}
Pick any two probability measures $\mu_1,\mu_2\in\Gamma_{\ell_0}$. Then for $n,m\in\N$
\begin{align*}
d_W({\bo \Delta} F_*^n(F_*^m\mu_1),{\bo \Delta} F_*^n(F_*^m\mu_2))&\le d_W({\bo {\mc L}}^n{\bo \Delta}F_*^m\mu_1,{\bo {\mc L}}^n{\bo \Delta}F_*^m\mu_2)+\\
&+d_W({\bo {\mc L}}^n{\bo \Delta}F_*^m\mu_1,{\bo \Delta} F_*^nF_*^m\mu_1)+\\
&+d_W({\bo {\mc L}}^n{\bo \Delta}F_*^m\mu_2,{\bo \Delta} F_*^nF_*^m\mu_2)\\
&\le  C_2d_{TV}({\bo {\mc L}}^n{\bo \Delta}F_*^m\mu_1,{\bo {\mc L}}^n{\bo \Delta}F_*^m\mu_2)+\\
&\quad\quad+2K_\#L^{n+1}(\sigma^{-1}+C_3(\sigma)\|\rho_m-\rho_g\|_{\infty})\\
&\le C_2C_{{\bo {\mc L}}}\lambda_{{\bo {\mc L}}}^n+2K_\#L^{n+1}(\sigma^{-1}+C_3(\sigma)\|\rho_m-\rho_g\|_{\infty}).
\end{align*}
where we used triangular inequality, Lemma \ref{Lem:dwlessdtv} (recall that $C_2$ is the diameter of $\T^{m_2}$), and Proposition \ref{Prop:ApproximDisitnMathcalL}.

For every $\epsilon>0$, pick $n_0\in\N$, $m_0\in\N$, and $\sigma_0$ large enough so that  $C_2C_{{\bo {\mc L}}}\lambda_{{\bo {\mc L}}}^{n_0}\le \epsilon/2$, and  
\[
2K_\#L^{n+1}(\sigma^{-1}+C_3(\sigma)\|\rho_{m_0}-\rho_g\|_{\infty})\le \frac{\epsilon}{2}
\]
Notice that $m_0$ is a transient one waits for the horizontal marginal to get sufficiently close  to $\nu_g$ while $n_0$ is the time one waits for $\bo{\mc L}$ to contract by the desired amount.

Calling  $C_{\bo{\mc L} }':=C_2C_{\bo{\mc L}}\lambda_{\bo{\mc L}}^{-m_0}$ 
\[
d_W({\bo \Delta} F_*^{n+m_0}\mu_1,{\bo \Delta} F_*^{n+m_0}\mu_2)\le C_2C_{{\bo {\mc L}}}\lambda_{{\bo {\mc L}}}^{n_0}+\epsilon/2\le C'_{{\bo {\mc L}}}\lambda_{{\bo {\mc L}}}^{n+m_0}+\epsilon.
\]
 for every $n$, in fact if
$n\le n_0$
\[
d_W({\bo \Delta} F_*^{n+m_0}\mu_1,{\bo \Delta} F_*^{n+m_0}\mu_2)\le C'_{{\bo {\mc L}}}\lambda_{{\bo {\mc L}}}^n+\epsilon/2
\]
and if $n\ge n_0$
\[
d_W({\bo \Delta} F_*^{n+m_0}\mu_1,{\bo \Delta} F_*^{n+m_0}\mu_2)\le C'_{{\bo {\mc L}}}\lambda_{{\bo {\mc L}}}^{n_0}+\epsilon/2\le C'_{{\bo {\mc L}}}\lambda_{{\bo {\mc L}}}^{n}+\epsilon.
\]

Recall that 
\[
\Pi F_*^{n+m_0}\mu_j=\sum_{i=1}^d{\bar\rho_i}({\bo \Delta} F_*^{n+m_0}\mu_j)_i
\]
and since the above is a convex combination, using Lemma \ref{Lem:ConvexComb}
\[
d_W(\Pi F_*^{n+m_0}\mu_1,\Pi F_*^{n+m_0}\mu_2) \le d_W ({\bo \Delta} F_*^{n+m_0}\mu_1,{\bo \Delta} F_*^{n+m_0}\mu_2)\le  C'_{{\bo {\mc L}}}\lambda_{{\bo {\mc L}}}^{n+m_0}+\epsilon
\]
which proves point i) with $t\ge m_0$. If $t\le m_0$, by the definition of $d_W$
\[
d_W(\Pi F_*^{m}\mu_1,\Pi F_*^{m}\mu_2)\le C_2
\]
 the diameter of $\T^{m_2}$. Therefore, picking 
\[
C''_{\bo{\mc L}}:=\max \left\{C'_{\bo{\mc L}}, \,C_2\lambda_{\bo{\mc L}}^{-m_0}\right\}
\] 
we get
\[
d_W(\Pi F_*^{t}\mu_1,\Pi F_*^{t}\mu_2) \le C''_{\bo{\mc L}}\lambda_{\bo{\mc L}}^t+\epsilon
\]
for all $t\in\N$ which concludes the proof of point i).

To prove point ii), recall that $\bo\eta_0\in(\mc M_1(\T^{m_2}))^d$ is fixed by $\bo{\mc L}$. Now
\begin{align*}
d_W({\bo \Delta} F_*^{n+m_0}\mu,\bo\eta_0)&\le d_W({\bo \Delta} F_*^{n+m_0}\mu, \bo{\mc L}^n\bo\Delta F_*^{m_0} \mu)+d_W(\bo{\mc L}^n\bo\Delta F_*^{m_0} \mu,\bo{\mc L}^n\bo\eta_0)\\
&\le d_W({\bo \Delta} F_*^{n+m_0}\mu, \bo{\mc L}^n\bo\Delta F_*^{m_0} \mu)+C_2d_{TV}(\bo{\mc L}^n\bo\Delta F_*^{m_0} \mu,\bo{\mc L}^n\bo \eta_0)\\
&\le  K_\#L^{n+1}(\sigma^{-1}+C_3(\sigma)\|\rho_m-\rho_g\|_{\infty})+C_2C_{\bo{\mc L}}\lambda_{\bo{\mc L}}^n.
\end{align*} 
In a way completely analogous to the proof of point i) one can show that for every $\epsilon>0$ there are $\sigma_0$ sufficiently large and $C_{\bo{\mc L}}''>0$  such that for $\sigma>\sigma_0$
\[
d_W({\bo \Delta} F_*^{t}\mu,\bo\eta_0)\le C''_{\bo{\mc L}}\lambda_{\bo{\mc L}}^{t}+\epsilon.
\]
By definition of $d_W$, the above means that $d_W(({\bo \Delta} F_*^t\mu)_i,(\bo\eta_0)_i)\le C_{\bo{\mc L}}''\lambda_{\bo{\mc L}}^t+\epsilon$ for all $i$, which implies that 
\[
d_W\left(\sum_i{\bar\rho_i}\,({\bo \Delta} F_*^t\mu)_i,\sum_i{\bar\rho_i}\,(\bo \eta_0)_i\right)\le C''_{\bo{\mc L}}\lambda_{\bo{\mc L}}^t+\epsilon.
\]
Since $\Pi F_*^n\mu=\sum_i{\bar\rho_i}\,({\bo \Delta} F_*^n\mu)_i$, the statement follows.
\end{proof}

\begin{proof}[Proof of Theorem \ref{Thm:AppDecayofCorr}]  With all the work above done, the proof of the theorem is almost identical to the case without  distortion. The only difference is that instead of $\eta_0$ in the proof of the case without distortion, one has to substitute $\sum_i{\bar\rho_i}\,(\bo \eta_0)_i$, and apply Proposition \ref{Prop:KantDistanceDistCase} in place of Proposition \ref{Prop:ContractionMarginalWassersteinDist}.

\end{proof}

\subsection{ Fixed point for ${ {\mc L}}$ and fixed point for $\mc P$.} \label{sec:difffixedpooint} Point ii) of Proposition \ref{Prop:KantDistanceDistCase} shows  that if $\sigma$ is sufficiently large, then  the vertical marginal of $F^n_* \mu$ becomes close to $\bar \eta=\sum_{i=1}^d{\bar\rho_i}(\bo \eta_0)_i$.   The purpose of this section is to remark that, in general, $\bar \eta$ is different (and possibly quite far) from $\eta_0$, the stationary measure of $\mc P$. We prove this fact  in an indirect way by showing that the unique fixed point of $\mc P:=\int_{\T}d\omega \rho_g(\omega)f_{\omega*}$, $\eta_0$, and the unique fixed point of  $\mc P':=\int_{\T}d\omega\rho_{g^{k-1}}(\omega)(f_{g^{k-1}(\omega)}\circ...\circ f_\omega)_*$, that we will call $\eta_0'$,  can be in general very different for some $k>1$. If this is the case, $\Pi F_*^{nk}\Leb_{\T^{m_1\times\T^{m_2}}}$ cannot become close to both $\eta_0$ and $\eta_0'$, and since $\mc P'$ is the random counterpart of $F^k$,  it implies that $\sum_{i=1}^d{\bar\rho_i}(\bo \eta_0)_i$ can be far from the fixed points of $\mc P$ or/and $\mc P'$.  At the end of the section we also give numerical evidence that $\bar \eta$ can be different from $\eta_0$ when the map $g$ has nonzero distortion.

 For simplicity of exposition, we are going to present an example that does not satisfy the smoothness requirements of Theorem \ref{Thm:AppDecayofCorr}. However, with a small modification on a set of arbitrarily small measure, the system can be made as smooth as one likes and all  the considerations below carry over to the smoothed version.
 
 First of all, we define the map {$g:= g_{M,\kappa}:\T\rightarrow \T$ where $M\in\N$ and $\kappa \in (0,1)$ are parameters}. We identify $\T$ with  $[0,1]$ in the usual way  and divide $[0,1]$ into $2M$ intervals of equal length
 \[
 I_j:=\left[\frac{j-1}{2M},\frac{j}{2M}\right].
 \]
 Let $\kappa'=1-\kappa$. Define for $0\le j\le M-1$ 
 \begin{equation}
 g_{M,\kappa}(\omega):=\left\{
 \begin{array}{ll}
 \frac{M}{\kappa}\omega-\frac{j}{2\kappa}& \omega\in [j/2M,(j+\kappa)/2M]\\
\frac{M}{\kappa'}\omega-\frac{j+\kappa-\kappa'}{2\kappa'} & \omega\in [(j+\kappa)/2M,(j+1)/2M]
 \end{array}
 \right.
 \end{equation}
 and for $M\le j\le 2M-1$ 
\begin{equation}
 g_{M,\kappa}(\omega):=\left\{
 \begin{array}{ll}
 \frac{M}{\kappa'}\omega-\frac{j}{2\kappa'}& \omega\in [j/2M,(j+\kappa')/2M]\\
\frac{M}{\kappa}\omega-\frac{j+\kappa'-\kappa}{2\kappa}  & \omega\in [(j+\kappa')/2M,(j+1)/2M]
 \end{array}
 \right.
 \end{equation}
The graph of $g_{5,0.99}$ is presented in Figure \ref{Fig:Basemap}.

\begin{figure}
\begin{center}
\includegraphics[scale=0.5]{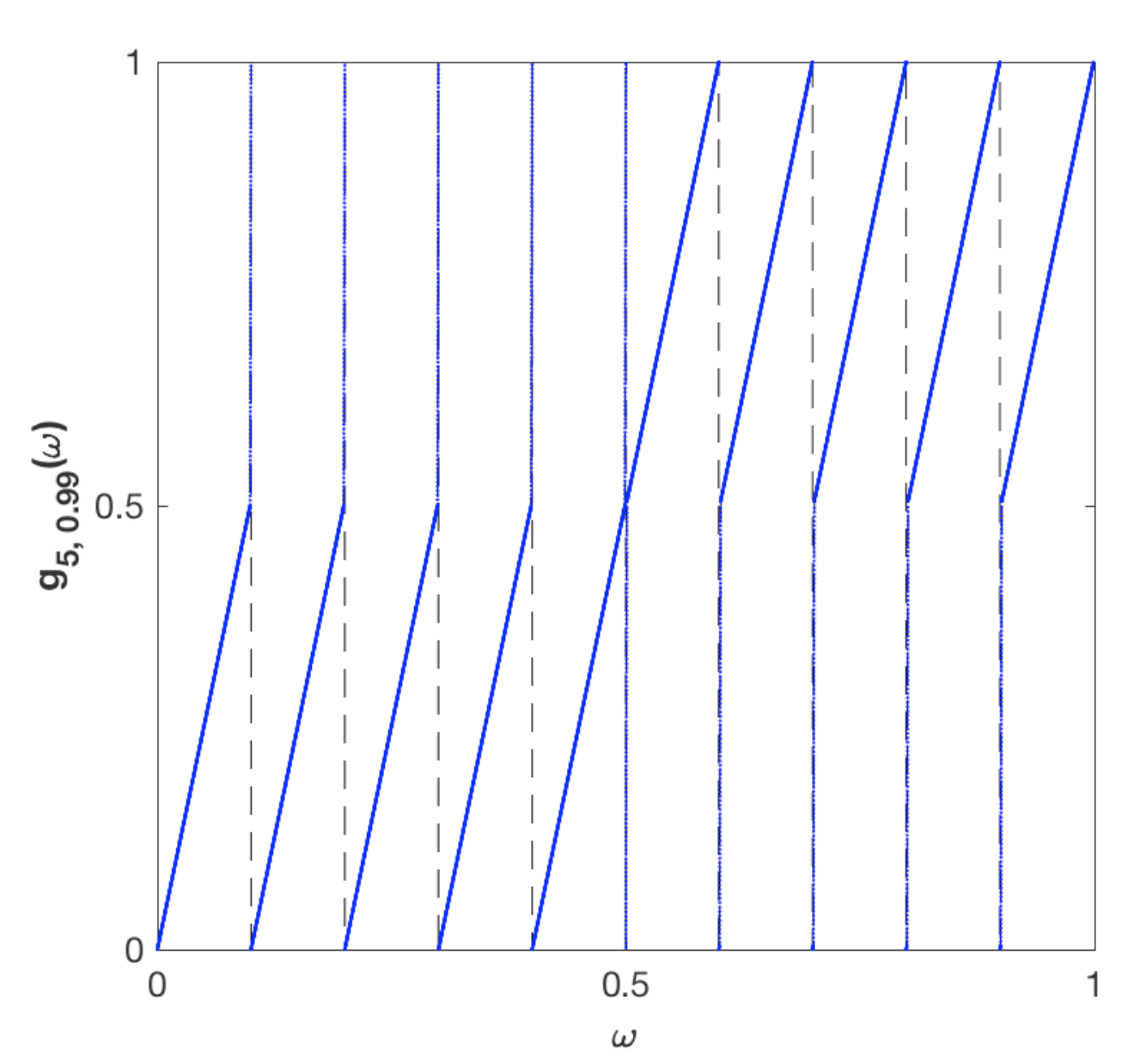}
\caption{Graph of $g_{5,0.99}$}
\label{Fig:Basemap}
\end{center}
\end{figure}

It is easy to verify that $g_{M,\kappa}$ is piecewise affine, uniformly expanding, and keeps the Lebesgue measure invariant. Also, the minimal expansion of $g_{M,\kappa}$ can be made arbitrarily large by letting $M\rightarrow\infty$. 

Notice that for $1\le j\le M$, 
\[
g_{M,\kappa}([j/2M,(j+\kappa)/2M])=[0,1/2] \,\,\mbox{ and }\,\, g_{M,\kappa}([(j+\kappa)/2M,(j+1)/2M])=[1/2,1]
\]
 while for $M+1\le j\le 2M$ 
 \[
 g_{M,\kappa}([j/2M,(j+\kappa')/2M])=[0,1/2]\,\,\mbox{ and }\,\,g_{M,\kappa}([(j+\kappa')/2M,(j+1)/2M])=[1/2,1].
 \] 
Picking $\kappa \approx 1$, most of the points in the interval $[0,1/2]$ are mapped back to $[0,1/2]$, and also most of the points of $[1/2,1]$ are  mapped back to $[1/2,1]$. More precisely, defining $V_1:=[0,1/2]$, $V_2:=[1/2,1]$ and
\[
V_{i,n}:=\{\omega\in V_i:\,g_{M,\kappa}^k(\omega)\in V_i\mbox{ for }0\le k\le n-1\};
\]
$V_{i,n}\subset V_i$ is such that, for any $n\in\N$, 
\begin{equation}\label{Eq:SizeofVi}
|V_{i,n}|\rightarrow 1/2\,\, \mbox{ as }\,\, \kappa\rightarrow 1.
\end{equation}

Fix $\epsilon>0$ a small number. Pick $\phi:\T\rightarrow\T$ a $N-S$ diffeomorphism such that $|\phi(x)-x|\le \epsilon$\footnote{A North-South (NS) diffeomorphism is a diffeomorphism with exactly two fixed points: one attracting, the South Pole (S), and one repelling, the North Pole (N), such that  for any $x\neq N$, $\phi^n(x)\rightarrow S$. Furthermore, $\phi'(N)>1$ and $\phi'(S)<1$ so that the two fixed points are hyperbolic.}, and define
 \begin{equation}\label{Eq:fibermaps}
 f_\omega(x)=\left\{
 \begin{array}{ll}
2\omega & \omega\in I_1\\
\phi+ a \omega &\omega\in I_2.
 \end{array}
 \right.
 \end{equation}
One can check that $\mc P:=\int_{\T}d\omega f_{\omega*}$ maps a small closed ball around $\Leb_\T$ into itself, with the diameter of the ball going to zero (in Total Variation distance) when $a\rightarrow 0$.  This implies that the unique fixed point of $\mc P$ is close to $\Leb_\T$. 


{To ease the notation, from now on we write $g$ in place of  $g_{M,\kappa}$.} Let's look at $f_{\omega}^{n-1}:=f_{g^{n-1}(\omega)}\circ...\circ f_\omega$ and study $\mc P':=\int_{\T}d\omega (f^{n-1}_{\omega})_*$. 
Fix $\Delta>0$ small. For any $\bar x_0\in\T$ and $(\omega_k)_{k=0}^{n-1}$ with $\omega_k\in V_2$, consider $(\bar x_k)_{k=0}^{n-1}$ with $\bar x_{k+1}=\phi(\bar x_k)+a\omega_k$. Pick $n\in\N$ large and $a>0$ small so that for any $\bar x_0\in[N-\Delta,N+\Delta]^c$ and $(\omega_k)_{k=0}^{n-1}$ as above,  $\bar x_{n-1}\in[S-\Delta,S+\Delta]$. One can find $\kappa$ close enough to one so that $|V_{1,n}|=|V_{2,n}|=0.49$, which implies  
\begin{align*}
\mc P'\eta&=\int_{V_{1,n}}d\omega (f_{\omega}^{n-1})_*\eta+\int_{V_{2,n}}d\omega (f_{\omega}^{n-1})_*\eta+\int_{(V_{1,n}\cup V_{2,n})^c}d\omega (f_{\omega}^{n-1})_*\eta\\
&= 0.49\Leb_\T+\int_{V_{2,n}}d\omega (f_{\omega}^{n-1})_*\eta+\int_{(V_{1,n}\cup V_{2,n})^c}d\omega (f_{\omega}^{n-1})_*\eta.
\end{align*}
 Given the expression of $\mc P'$, if $\eta_0'$ is such that $\mc P'\eta_0'=\eta_0'$ then, $\eta_0'= 0.49 \Leb+0.51\eta_1$, where $\eta_1$ is some probability measure. This implies that 
\begin{align*}
\mc P'\eta_0([S-\Delta,S+\Delta])&=(0.49\mc P'\Leb+0.51\mc P'\eta_1)([S-\Delta,S+\Delta])\\
&=0.49\int_{V_{2,n}}d\omega (f_{\omega}^{n-1})_*\Leb([S-\Delta,S+\Delta])+\\
&\quad\quad\quad\quad+(1-0.49^2)\eta_2([S-\Delta,S+\Delta])\\
&>0.49^2(1-2\Delta)
\end{align*}
where $\eta_2$ above is some probability measure.  Since $\Delta>0$ is arbitrary, $\eta_0'([S-\Delta,S+\Delta])\approx 1/4$ while $\eta_0([S-\Delta,S+\Delta])\approx 2\Delta$ which makes $\eta_0$ and $\eta_0'$ two  very far apart measures with respect to most metrics (e.g. $d_{TV}$, $d_W$,...).

\begin{figure}[h!]
\begin{center}
\includegraphics[scale=0.6]{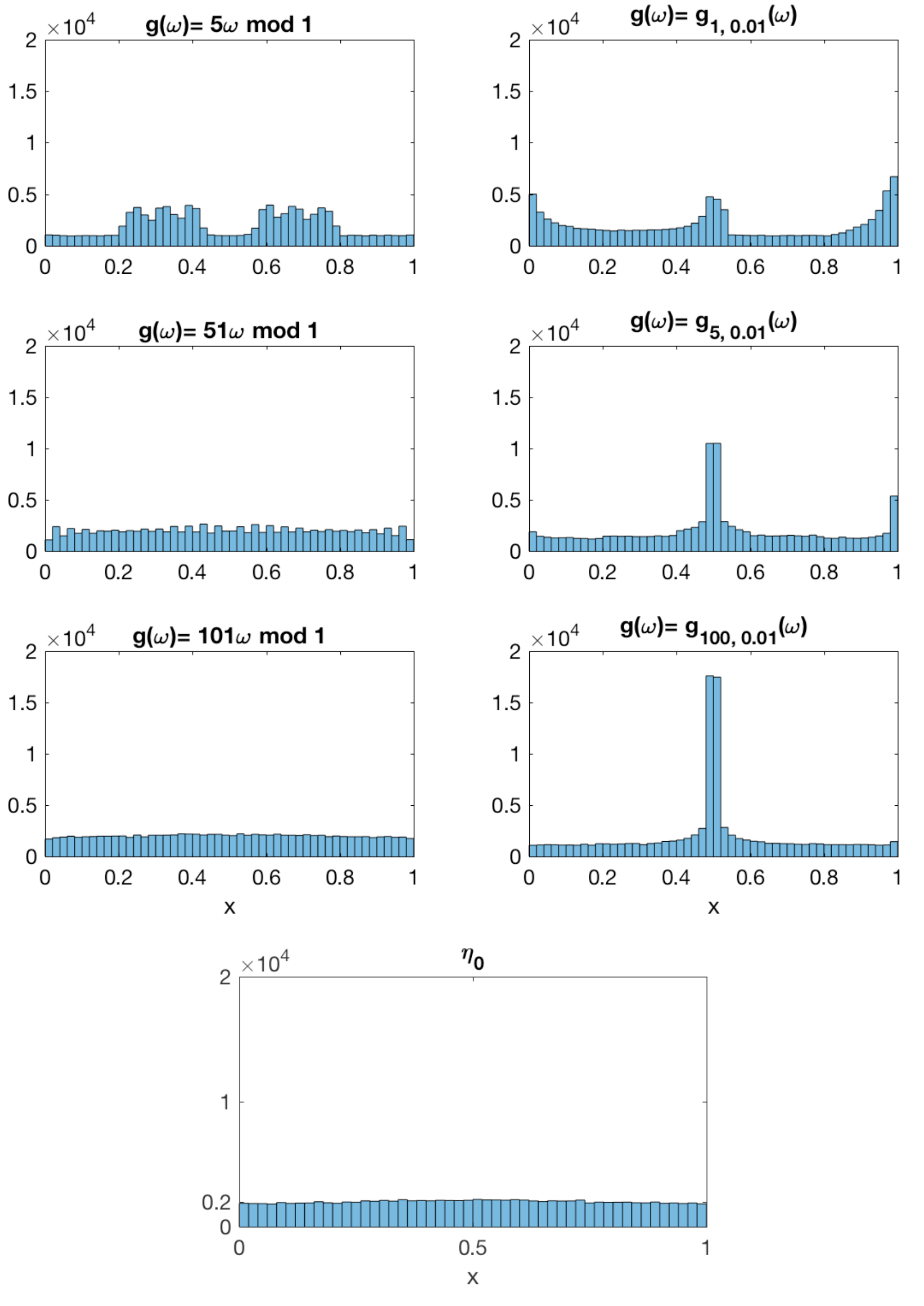}
\end{center}
\caption{For different base maps $g$, we consider $10^4$ initial conditions $\{(\omega_k,x_k)\}_{k=1}^{10^4}$ sampled randomly and uniformly on $[0,1]\times [0,1]$, let $F(\omega,x)=(g(\omega),f(\omega,x))$ act for 100 time steps to obtain $\{F^{100}(\omega_k,x_k)\}_{k=1}^{10^4}$, take the vertical $x$-coordinates of these points, and plot them on a histogram. The different $g$ maps used are indicated above the histograms. The fiber maps are the same throughout and as in \eqref{Eq:fibermaps} with $\phi(x)=x-0.01\sin(2\pi x)$ and $a=0.001$. 
The last panel shows a numerical approximation for $\eta_0$ obtained as in the deterministic case by applying $F$ to $\{(\omega_k,x_k)\}_{k=1}^{10^4}$ but where, instead of having $g$ in the base, we sampled the $\omega$-coordinate at random independently (both w.r.t. time and  initial conditions) and uniformly on $[0,1]$ using the random number generator built in the programming language.}
\label{Fig:Instograms}
\end{figure}

{In Figure \ref{Fig:Instograms} below we compare numerical simulations of the distribution of mass on the vertical marginal after several iterations of  skew-products $F$ with different base maps $g$. For each such map, we consider several initial conditions sampled randomly and uniformly on $[0,1]\times [0,1]$, let $F$ act for a while on these points, then take their vertical coordinates, and plot them on a histogram.  When the expansion in the base is large, we expect the distribution given by the histogram to be close to $\bar \eta$.
We compare the case of base maps with no distortion,  $g(\omega)=\sigma\omega$ mod 1, against base maps $g_{M,\kappa}$ defined above.  We also simulate numerically $\eta_0$, the stationary measure for $\mc P$ (as given by the random number generator of the programme). The fiber maps $f_\omega$ are of the kind described in \eqref{Eq:fibermaps}. 

In the case without distortion, when the minimal expansion in the base increases, we can see that the simulated $\bar \eta$  becomes very close to $\eta_0$  (as per Propostion \ref{Prop:ContractionMarginalWassersteinDist} point ii)), while in the case with distortion, $\bar \eta$ and $\eta_0$ are different.}

\section{Generalizations and limitations }\label{Sec:Generalizations}

In this section we discuss a few generalizations of the results and techniques presented above, and also some of the limitations. Before proceeding with the generalizations, we would like to stress that the goal of this paper was not to give a result in its greatest generality possible, but rather to present some techniques that we believe can be applied (with different levels of additional effort) to various setups. 

\subsection{Regularity assumptions on $g$.}\label{Sec:GenRegg} The regularity assumptions on the map $g$ can be revised to fit other situations.  For example, $\Omega$ could be a compact manifold with border such as $\Omega=[0,1]^{m_1}$ with $g$  piecewise $C^2$ with onto branches. By this we mean that there are open sets $\{I_i\}_{i=1}^d$ partitioning $\Omega$ modulo sets of measure zero, and such that $g|_{I_i}:I_i\rightarrow (0,1)^{m_1}$ is a $C^2$ uniformly expanding diffeomorphism with bounded distortion. 

For the system in Section \ref{Sec:NoDistortion}, i.e. when $m_1=1$ and no distortion, this corresponds to considering  maps $g:[0,1]\rightarrow [0,1]$ for which there are $n\in\N$ and $0=:a_0<a_1<...<a_n<a_{n+1}:=1$ such that  $g|_{(a_i,a_{i+1})}$ is $C^2$ and onto $(0,1)$. It is easy to check that all the proof of statements in Section \ref{Sec:NoDistortion} hold, \emph{mutatis mutandis}, for maps $g$  satisfying these assumptions. 

Also the assumption that $g$ must be $C^2$ (or piecewise $C^2$) is not necessary, and can be substituted by $g$ being $C^{1+\alpha}$ (or piecewise $C^{1+\alpha}$), meaning that $g$ is once differentiable and with $\alpha-$H\"older differential (or same property, but piecewise).

\subsection{Robustness under  conjugacy} 
 Consider a map $\hat g:\Omega \rightarrow \Omega$ and assume that there is an invertible map $h:\Omega \rightarrow \T^{m_1}$ which is measurable and with measurable inverse, such that $\hat g:=h^{-1}\circ g\circ h$ for a map $g:\T^{m_1}\rightarrow \T^{m_1}$. Consider $\hat f:\Omega\times\T^{m_2}\rightarrow \T^{m_2}$ and the skew-product system $\hat F:\Omega\times \T^{m_2}\rightarrow \Omega\times\T^{m_2}$
 \[
 \hat F(\omega, x)=(\hat g(\omega),\hat f(\omega,x) ).
 \]
 Then, if one can show that the skew-system $F:\T^{m_1}\times \T^{m_2}\rightarrow \T^{m_1}\times \T^{m_2}$
 \[
 F(\omega,x)=(g(\omega), f(\omega,x))
 \]
 with $f(\omega,x):=\hat f(h^{-1}\omega,x)$, satisfies an approximate decay of correlations (as in Theorem \ref{Thm:AppDecayofCorr}), then so does $\hat F$. This is made precise in the following proposition.
\begin{proposition}
Suppose $\hat F:\Omega\times \T^{m_2}\rightarrow \Omega\times\T^{m_2}$ and $F:\T^{m_1}\times \T^{m_2}\rightarrow \T^{m_1}\times \T^{m_2}$ are as above, and assume that 
for some $\epsilon>0$, $\eta$ a probability measure, $\tilde C>0$ and $\tilde \lambda\in (0,1)$  the conclusion of Theorem \ref{Thm:AppDecayofCorr} holds for $F$.  
Then, defining $\nu:=(h^{-1})_*\Leb_{\T^{m_1}}$
\[
\left|\int_{\Omega\times\T^{m_2}} \phi( \pi_2\hat F^n(\omega,x)) \psi(x)dxd\nu(\omega)-\int_{\T^{m_2}} \phi(x)d\eta(x)\int_{\T^{m_2}}\psi(x)dx\right|\leq C_{\phi,\psi}(\tilde C \tilde{\lambda}^n+ \epsilon) 
\]
for all $\psi\in L^1(\T^{m_2};\R)$ and $\phi\in\Lip(\T^{m_2};\R)$.
\end{proposition}
\begin{proof}
Take $\psi\in L^1(\T^{m_2};\R)$ and $\phi\in\Lip(\T^{m_2};\R)$. Define $\nu=(h^{-1})_*\Leb_{\T^{m_1}}$ a probability measure on $\Omega$. Let's call $H:=h \times \id$ which is invertible with inverse $H^{-1}=h^{-1}\times \id$.
\begin{align*}
\int_{\Omega\times\T^{m_2}}\psi\; \phi\circ \pi_2\circ \hat F^n d\nu\otimes \Leb&=\int_{\Omega\times\T^{m_2}}\psi\circ\pi_2\circ H^{-1}\;\phi\circ\pi_2\circ\hat F^n\circ H^{-1} dH_*( \nu\otimes \Leb)\\
&=\int_{\Omega\times\T^{m_2}}\psi\;  \phi \circ\pi_2H\circ\hat F^n\circ H^{-1} d\Leb\\
&=\int_{\Omega\times\T^{m_2}}\psi\;    \phi\circ\pi_2 \circ F^n d\Leb.
\end{align*}

Therefore, from the assumptions, there is $C_{\phi,\psi}>0$ such that 
\begin{align*}
\left|\int_{\Omega\times\T^{m_2}}\psi\; \phi\circ \pi_2\circ \hat F^n d\nu\otimes \Leb-\int_{\T^{m_2}} \phi(x)d\eta(x)\int_{\T^{m_2}}\psi(x)dx\right|\le C_{\phi,\psi}(\tilde C\tilde \lambda^n+\epsilon)
\end{align*}
\end{proof}

As an example, one can use Theorem \ref{Thm:AppDecayofCorr} to prove approximate decay of correlation in case the forcing is driven by a power of the logistic map $\hat g_0(x)=4x(1-x)$. In fact, it is well known $\hat g_0$ is conjugate to the tent map
\[
g_0=\left\{
\begin{array}{ll}
2x & x\in[0,1/2)\\
1-2x& x\in[1/2,1]
\end{array}\right.
\]
via a $C^1$ map $h:[0,1]\rightarrow [0,1]$. 
Analogously, for any $n\in\N$, also $\hat g:=\hat  g_0^n$ is conjugate to $g:=g_0^n$ via $h$, and $g$ is in the class of maps admitted by the generalization in Section \ref{Sec:GenRegg} for which one can apply Theorem \ref{Thm:AppDecayofCorr}.

\subsection{More or less regular disintegrations} In Definition \ref{Def:HolderDisint} we have given the definition of Lipschitz disintegration $\{\mu_\omega\}_{\omega\in\Omega}$ and later we have shown how, under the hypotheses of Theorem \ref{Thm:AppDecayofCorr}, certain classes of measures with Lipschitz disintegration were kept invariant by the dynamics. Measures having  H\"older disintegration can be defined in a completely analogous way, and they can be used to define classes of invariant measures for example in the case where $f:\Omega\times X\rightarrow X$ is only H\"older and not Lipschitz. 

Analogously, one could think of defining measures having disintegrations of higher regularity, e.g. differentiable for a suitable notion of differentiability for curves in $\mc M_1(X)$, and exploit these classes. 

\subsection{Limitations of the approach}
{A more substantial and also natural step forward from Theorem \ref{Thm:AppDecayofCorr}, would be considering $g$ an invertible uniformly hyperbolic map, like an Anosov diffeomorphism or a map with an Axiom A attractor. Unfortunately it seems hard to extend the techniques in this paper to this case. The main reason is that we need the contraction properties of the inverse branches of $g$: for invertible uniformly hyperbolic systems, some directions are contracted when taking preimages, but others are expanded and this spoils the arguments.  

For the same reason our approach is evidently ill-suited to treat skew-products with quasi-periodic base (see e.g. \cite{dolgopyat2015limit}).}

\appendix
\section{Markov chains and random dynamics}\label{App:MarkChainsAndRandom}

In this section we report some classical results about geometric ergodicity of Markov chains {(\cite{DownMeynTweedie, MeynTweedieBook, MM11, CM15})}, and we relate this to random dynamical systems in discrete time.

\subsection{Markov chains and geometric ergodicity}\label{App:MarkGeomErg}
\begin{definition}
Given a Polish space $ S$, the \emph{state space}, endowed with a countably generated $\sigma-$field $\mc B(S)$, a discrete time Markov process is a sequence of random variables $\{X_t\}_{t\in\N_0}$ defined on a probability space $(\Omega,\Sigma, \mb P)$ such that for all $n\in\N$ 
\[
\mb E_{\mb P}[X_n|X_{n-1},\,...,\, X_0]=\mb E_{\mb P}[X_{n}|\, X_{n-1}].
\]

The Markov process is called \emph{stationary} if $\mb E_{\mb P}[X_{n}|\, X_{n-1}]$ does not depend on $n$, and $P:S\times \mc B(S)\rightarrow \R^+$ is the associated transition kernel satisfying
\[
\mb P\left(X_{n+1}\in A|\, X_{n}=x\right)=P(x,A).
\]
\end{definition}

For every $x\in S$, $P(x,\cdot)$ defines a probability measure with the following meaning: $P(x,A)$ is the probability that $X_{n+1}\in A$ given that at time $n$ one has observed $X_n=x$.

Given a stationary Markov process and $n\in\N$, one can extend the notion of kernel to higher iterates: Define $P^m:S\times \mc B(S)\rightarrow \R^+$ 
\[
P^m(x,A)=\mb P\left(X_{n+m}\in A|\, X_{n}=x\right).
\]
For any $n\in\N$, $P^n$ generates an action on the set of measures on $(S,\mc B(S))$ in the following way. Given $\mu$ a measure on $S$, define
\[
\mc P\mu(A):=\int_{S}P(x, A)d\mu(x).
\]
and
\[
\mc P^n\mu(A):=\int_{S}P^n(x, A)d\mu(x).
\]
Using the properties of transition kernels one can prove that, $\{\mc P^n\}$ satisfies the semi-group property
\[
\mc P^n\circ \mc P^m=\mc P^{n+m}
\]
making $\mc P$ the generator of a semi-group action on probability measures on $(S,\mc B(S))$.

\begin{definition}\label{Def:GeomErg}
A stationary Markov chain is said to be \emph{geometrically ergodic} if there are $C>0$ and $\lambda\in(0,1)$ such that 
\[
d_{TV}(P^n(x_1,\cdot),P^n(x_2,\cdot))\le C\lambda^n,\quad\quad\forall x_1,x_2\in S.
\]
\end{definition}
For a definition of the Total Variation distance $d_{TV}$ see the beginning of Sec. \ref{App:Wasserstein}. From  Definition \ref{Def:GeomErg} follows that if a Markov chain is \emph{geometrically ergodic}, then there is a probability measure $\eta_0$  such that, for every probability measure $\mu$ on $(S,\mc B(S))$,
\[
d_{TV}(\mc P^n\mu,\eta_0)\le C\lambda^n.
\]
The measure $\eta_0$ satisfies $\mc P(\eta_0)=\eta_0$ and is also called a \emph{stationary distribution} or \emph{stationary measure}.

\subsection{Sufficient conditions for geometric ergodicity} In this subsection we give a  sufficient condition that ensures geometric ergodicity of a stationary Markov chain. Weaker conditions working in more general setups are available and involve \emph{petite sets} \cite{DownMeynTweedie} or Lyapunov functions \cite{MM11}. 

\begin{theorem}[{ \cite{MeynTweedieBook}}]\label{Thm:GeometricErgodicity}
Let $\{X_n\}_{n\in\N_0}$ be a stationary Markov chain on $(S,\mc B(S))$ with transition kernel $P:S\times \mc B(S)\rightarrow \R^+$. Assume there is $\nu$ a probability measure, $\epsilon>0$ and $n_0\in\N$ such that 
\[
P^{n_0}(x,\cdot)\ge \epsilon \nu(\cdot),\quad\quad\forall x\in S.
\] 
Then the Markov chain is geometrically ergodic.
\end{theorem}

\subsection{Randomly forced systems and Markov chains}\label{App:RandSysMarkChains}

In this section we discuss the difference, in terms of mathematical definitions, between random and deterministic forcing. 

By \emph{random forcing}, we mean that given a probability space  $(\Omega,\nu)$ and  $f:\Omega\times X\rightarrow X$, at the $n-$th iteration we apply the map $f_{\zeta_n}:=f(\zeta_n,\cdot):X\rightarrow X$, where $\{\zeta_n\}_{n\in\N}$ is an i.i.d  sequence of random variables defined on some probability space $(\Xi,\mb P)$ with values in $\Omega$ and distributed according to $\nu$. Fixed $w\in\Xi$, the forward orbits of the system are given by 
\[
O(x):=\left\{f_{\zeta_n(w)}\circ...\circ f_{\zeta_1(w)}\circ f_{\zeta_0(w)}(x): \, n\in\N_0\right\}, \quad\quad\forall x\in X.
\]

An important example of random forcing is given by additive i.i.d. noise: Consider $X=\T^m$, or any other set with an additive structure,  a map $T:X\rightarrow X$
and $\{\zeta_n\}_{n\in\N_0}$ an i.i.d. sequence of random variables with values in $X$ and distributed according to $\nu$, then taking $\Omega=X$ define $f:\Omega\times X\rightarrow X$ as
\[
f(\omega,x):=T(x)+\omega.
\]
Composing at time  $n\in\N$ by $f_{\zeta_n}$ corresponds to considering the recursive equation
\[
\mc X_{n+1}= T(\mc X_n)+\zeta_n,\quad\quad\forall n\in\N_0.
\]
where $\mc X_{n}$ is the state of the system at time $n$. 
What the above means is that, calling $(\Xi,\mb P)$ the underlying probability space where $\{\zeta_n\}_{n\in\N_0}$ are defined,  $\{\mc X_n\}_{n\in\N_0}$ are random variables satisfying
\[
\mb P\left(\mc X_{n+1}\in A|\mc X_{n}=x_n\right)=\mb P(\xi_n\in (A-T(x_n))),
\] 
and thus $\{\mc X_n\}_{n\in\N_0}$ is a Markov chain.
In the above, $T$ denotes the intrinsic dynamics, i.e. the dynamics the system would have if it did not receive any forcing, while $\xi_n$ is the random forcing noise term.

\emph{Deterministic forcing} is also represented as application at time $n\in\N$ of the map $f(\zeta_n,\cdot):X\rightarrow X$. However, in this case the sequence $\{\zeta_n\}_{n\in\N}$ is not required to be independent, but it should satisfy $\zeta_{n+1}=g(\zeta_n)$ for some transformation $g:\Omega\rightarrow \Omega$ that preserves the measure $\nu$. This corresponds also to the general definition of \emph{random dynamical system} usually given in the literature (see \cite{Arnold}).

The difference between random and deterministic forcing is not a stark one. In fact one can  show that random forcing is a particular case of deterministic forcing where $g$ is an appropriate shift map. In fact given $f:\Omega\times X\rightarrow X$ and $\{\zeta_n\}_{n\in\N_0}$ an i.i.d sequence with values in $\Omega$ distributed as $\nu$, we can construct the probability space $(\Omega',\nu')$ with $\Omega':=\Omega^{\N_0}$ and $\nu':=\nu^{\otimes \N_0}$. Now define the sequence of identically distributed random variables $\{\bo \zeta_k\}_{k\in\N_0}$  in  $\Omega'$  with $\bo \zeta_k:=\{\zeta_{n+k}\}_{n\in\N_0}$, and $f':\Omega'\times X\rightarrow X$
\[
f'(\bo \zeta,x):= f((\bo \zeta)_0, x)
\]
where $(\bo \zeta)_0$ denotes the first term of the sequence $\bo\zeta\in\Omega'$.  With this definition we also have 
\[
\bo\zeta_{k+1}=\sigma^{k+1}(\{\zeta_n\}_{n\in\N_0})=\sigma( \bo \zeta_{k-1})
\]
 where $\sigma:\Omega'\rightarrow \Omega'$ is the left shift which is easy to check that keeps the measure $\nu'$ invariant.

\section{Disintegration of measure and Rohlin's theorem}\label{App:RohlinThm}
The following definitions and results are taken from \cite{Simmons}, adapted to the level of generality needed in this paper. 
\begin{definition}\label{Def:ConditionalMeas}
Let $(X,\mu)$ be a topological probability space, $Y$ a metric space and $\pi:X\rightarrow Y$ a measurable function. Call $\hat \mu:=\pi_*\mu$.  A system of conditional measures of $\mu$  with respect to $(X,\pi, Y)$ is a collection of measures $\{\mu_y\}_{y\in Y}$ such that 
\begin{itemize}
\item[1)] For $\hat\mu-$almost every $y\in Y$, $\mu_y$ is a probability measure on $\pi^{-1}(y)$.
\item[2)] For every measurable subset $B\subset X$, $y\mapsto \mu_y(B)$ is measurable and
\[
\mu(B)=\int \mu_{\pi^{-1}(y)}(B)d\hat\mu(y).
\]
\end{itemize}
\end{definition}

When $Y$ in the above definition is a measurable partition of  $X$ and $\pi(x)$ is the unique element of the partition to which $x$ belongs, then we also call $\{\mu_y\}_{y\in Y}$ a disintegration of $\mu$. 

\begin{definition}\label{Def:TopCondMeas}
In the same setup of Definition \ref{Def:ConditionalMeas}, the topological conditional measure of $\mu$ with respect to $(X,\pi,y,Y)$ is the weak$*$ limit (if it exists)
\[
\mu_y:=\lim_{\epsilon\rightarrow 0^+}\mu_{\pi^{-1}(B(y,\epsilon))}
\]
where $B(y,\epsilon)$ is the ball centered at $y$ with radius $\epsilon$ with respect to the metric on $Y$ and
\[
\mu_{\pi^{-1}(B(y,\epsilon))}(I)=\frac{\mu(\pi^{-1}(B(y,\epsilon))\cap I)}{\mu(\pi^{-1}(B(y,\epsilon)) }.
\]
\end{definition}

\begin{theorem}[Theorem 2.2 \cite{Simmons}]\label{Thm:DisintSimm}
Let $(X,\mu)$ be a compact metric probability space, let $Y$ be a separable Riemannian manifold. Let $\pi:X\rightarrow Y$ be measurable.Then for $\hat \mu-$almost every $y\in Y$, the topological conditional measure of $\mu$ with respect to $(X,\pi, y, Y)$ exists as in Definition \ref{Def:TopCondMeas}. Furthermore the collection of measures  $\{\mu_y\}_{y\in Y}$ is a system of conditional measures as in Definition \ref{Def:ConditionalMeas}. (If $\mu_y$ does not exist, set $\mu_y=0$).
\end{theorem}

\section{Wasserstein distance: some computations}\label{App:Wasserstein}
Consider a compact metric space $(Y,d)$. Then the Kantorovich-Wasserstein between $\mu_1,\mu_2\in\mc M_1(Y)$ is defined as 
\[
d_W(\mu_1,\mu_2):=\sup_{\gamma\in\mc C(\mu_1,\mu_2)}\int_{Y\times Y}d(s,s')d\gamma(s,s')
\] 
where $\mc C(\mu_1,\mu_2)$ is the set of all couplings between $\mu_1$ and $\mu_2$. If we consider instead of the metric $d$ the discrete metric  $d_{dis}$ defined as 
\[
d_{dis}(s,s')=\left\{
\begin{array}{ll}
1&s=s'\\
0&s\neq s'
\end{array}\right.
\]
We have that 
\[
d_{TV}(\mu_1,\mu_2):=\sup_{\gamma\in\mc C(\mu_1,\mu_2)}\int_{Y\times Y}d_{dis}(s,s')d\gamma(s,s').
\]

\begin{lemma}\label{Lem:LipcConstPushForw}
Let $(Y,d)$ be a metric space, $T:Y\rightarrow Y$ a Lipschitz transformation with Lipschitz constant $\Lip(T)$. Then for any $\xi_1,\xi_2\in\mc M_1(\T)$
\[
d_W(T_*\xi_1,T_*\xi_2)\le \Lip(T)d_W(\xi_1,\xi_2).
\]
\end{lemma}
\begin{proof}
\begin{align*}
d_W(T_*\xi_1,T_*\xi_2)&=\sup_{\phi\in\Lip^1(Y)}\int_Y\phi(y)d(T_*\xi_1-T^*\xi_2)(y)\\
&=\sup_{\phi\in\Lip^1(Y)}\int_Y\phi\circ Td(\xi_1-\xi_2)(y)\\
&\le \sup_{\phi\in\Lip^1(Y)} \Lip(\phi\circ T) d_W(\xi_1,\xi_2)
\end{align*}
and $\Lip(\phi\circ T)=\Lip(\phi)\Lip(T)$.
\end{proof}
\begin{remark}
The above lemma can be read in the following way: If $T:(Y,d)\rightarrow (Y,d)$ is Lipschitz, then $T_*:(\mc M_1(Y),d_W)\rightarrow (\mc M_1(Y),d_W)$ is Lipschitz with $\Lip(T_*)=\Lip(T)$.
\end{remark}

\begin{lemma}\label{Lem:AvofMeasures}
Consider $(S,\nu)$ a measurable space with $\nu$ a probability measure, and $Y$ a compact metric space. Assume that $\{\mu_s\}_{s\in S}$ is a family of measures belonging to $\mc M_1(Y)$ and that  $\exists \ell>0$ s.t.  $d_W(\mu_s,\mu_{s'})\le \ell$ for for every $s,s'\in S$. Then the measure $\bar\mu\in\mc M_1(Y)$ defined as
\[
\bar\mu(A):=\int_{S}d\nu(s)\mu_s(A)
\]
is such that $d_W(\bar\mu,\mu_s)\le \ell$ for all $s\in S$.
\end{lemma}
\begin{proof}
Pick $s\in  S$
\begin{align*}
d_W(\bar \mu,\mu_s)&=\sup_{\phi\in\Lip^1(Y)}\int_Y\phi(y)d(\bar\mu-\mu_s)(y)\\
&= \sup_{\phi\in\Lip^1(Y)}\int_Y\int_Sd\nu(s')\phi(y)d(\mu_{s'}-\mu_s)(y)\\
&\le \int_Sd\nu(s')\sup_{\phi\in\Lip^1(Y)}\int_Y\phi(y)d(\mu_{s'}-\mu_s)(y)\\
&\le \int_Sd\nu(s')d_W(\mu_s,\mu_{s'})\\
&\le \ell.
\end{align*}
\end{proof}

\begin{lemma}\label{Lem:dwlessdtv}
Let $(Y,d)$ be a bounded metric space and call $\diam (Y)$ its diameter. Then
\[
d_W(\mu_1,\mu_2)\le \diam(Y) d_{TV}(\mu_1,\mu_2).
\]
\end{lemma}
\begin{proof}
\begin{align*}
d_W(\mu_1,\mu_2)&=\sup_{\gamma\in\mc C(\mu_1,\mu_2)}\int_{Y\times Y}d(s,s')d\gamma(s,s')\\
&\le \sup_{\gamma\in\mc C(\mu_1,\mu_2)}\int_{Y\times Y}\diam(Y)d_{dis}(s,s')d\gamma(s,s')\\
&=\mbox{diam}(Y) d_{TV}(\mu_1,\mu_2).
\end{align*}
\end{proof}

\begin{lemma}\label{Lem:ConvexComb}
Assume $\{\mu_i\}_{i=1}^n$ and $\{\mu_i'\}_{i=1}^n$ are probability measures in $\mc M_1(Y)$ and $\{b_i\}_{i=1}^n$, $b_i>0$, are weights with $\sum_{i=1}^nb_i=1$. Then
\[
d_W(\sum_{i=1}^nb_i\mu_i,\sum_{i=1}^nb_i\mu_i')\le \max_i d_W(\mu_i,\mu_i').
\]
\end{lemma}
\begin{proof}
\begin{align*}
d_W(\sum_{i=1}^nb_i\mu_i,\sum_{i=1}^nb_i\mu_i')&\le \sup_{\phi\in\Lip^1}\int_Y \phi d\left(\sum_{i=1}^nb_i\mu_i-\sum_{i=1}^nb_i\mu_i'\right)\\
&\le \sum_{i=1}^nb_i\sup_{\phi\in\Lip^1}\int_Y\phi d(\mu_i-\mu_i')\\
&\le \sum_{i=1}^nb_i d_W(\mu_i,\mu_i')\\
&\le \max_id_W(\mu_i,\mu_i').
\end{align*}
\end{proof}

\bibliographystyle{amsalpha} 
\bibliography{bibliografia.bib}

\providecommand{\bysame}{\leavevmode\hbox to3em{\hrulefill}\thinspace}
\providecommand{\MR}{\relax\ifhmode\unskip\space\fi MR }
\providecommand{\MRhref}[2]{%
  \href{http://www.ams.org/mathscinet-getitem?mr=#1}{#2}
}
\providecommand{\href}[2]{#2}
\begin{thebibliography}{DFGTV20}

\bibitem[ABV00]{ABV00}
Jos\'e~F. {Alves}, Christian {Bonatti}, and Marcelo {Viana}, \emph{{SRB
  measures for partially hyperbolic systems whose central direction is mostly
  expanding}}, {Invent. Math.} \textbf{140} (2000), no.~2, 351--398 (English).

\bibitem[ADLP17]{alves2017srb}
Jos{\'e}~F Alves, Carla~L Dias, Stefano Luzzatto, and Vilton Pinheiro,
  \emph{{SRB} measures for partially hyperbolic systems whose central direction
  is weakly expanding}, Journal of the European Mathematical Society
  \textbf{19} (2017), no.~10, 2911--2946.

\bibitem[{Arn}98]{Arnold}
Ludwig {Arnold}, \emph{{Random dynamical systems}}, Berlin: Springer, 1998
  (English).

\bibitem[BE17]{ButterleyEslami}
Oliver {Butterley} and Peyman {Eslami}, \emph{{Exponential mixing for skew
  products with discontinuities.}}, {Trans. Am. Math. Soc.} \textbf{369}
  (2017), no.~2, 783--803 (English).

\bibitem[BG97]{BG97}
Abraham {Boyarsky} and Pawe{\l} {G\'ora}, \emph{{Laws of chaos. Invariant
  measures and dynamical systems in one dimension.}}, Boston, MA: Birkh\"auser,
  1997 (English).

\bibitem[Bje18]{B18}
Kristian Bjerkl\"ov, \emph{A note on circle maps driven by strongly expanding
  endomorphisms on $t$}, Dynam.\ Sys. \textbf{33} (2018), 361--368.

\bibitem[BW10]{WB2010}
Keith {Burns} and Amie {Wilkinson}, \emph{{On the ergodicity of partially
  hyperbolic systems}}, {Ann. Math. (2)} \textbf{171} (2010), no.~1, 451--489
  (English).

\bibitem[BY93]{BaladiYoung}
Viviane. Baladi and Lai-Sang. Young, \emph{On the spectra of randomly perturbed
  expanding maps}, Comm. Math. Phys. \textbf{156} (1993), no.~2, 355--385.

\bibitem[CFKM20]{CFKM20}
Ilya {Chevyrev}, Peter~K. {Friz}, Alexey {Korepanov}, and Ian {Melbourne},
  \emph{{Superdiffusive limits for deterministic fast-slow dynamical systems}},
  {Probab. Theory Relat. Fields} \textbf{178} (2020), no.~3-4, 735--770
  (English).

\bibitem[CL20]{CastorriniLiverani20}
Roberto Castorrini and Carlangelo Liverani, \emph{Quantitative statistical
  properties of two-dimensional partially hyperbolic systems},
  https://arxiv.org/abs/2007.05602 (2020).

\bibitem[Clo15]{CM15}
Martin Cloez, Bertrand;~Hairer, \emph{Exponential ergodicity for {M}arkov
  processes with random switching}, Bernoulli \textbf{21} (2015).

\bibitem[CM00]{BV00}
Bonatti Christian and Viana Marcelo., \emph{{SRB} measures for partially
  hyperbolic systems whose central direction is mostly contracting}, Isr. J.
  Math. \textbf{115} (2000), 157--193.

\bibitem[DF15]{dolgopyat2015limit}
Dmitry Dolgopyat and Bassam Fayad, \emph{Limit theorems for toral
  translations}, Hyperbolic dynamics, fluctuations and large deviations
  \textbf{89} (2015), 227--277.

\bibitem[DFGTV18]{dragivcevic2018spectral}
Davor Dragi{\v{c}}evi{\'c}, Gary Froyland, Cecilia Gonzalez-Tokman, and Sandro
  Vaienti, \emph{A spectral approach for quenched limit theorems for random
  expanding dynamical systems}, Comm. Math. Phys. \textbf{360} (2018), no.~3,
  1121--1187.

\bibitem[DFGTV20]{dragivcevic2020spectral}
Davor Dragi{\v{c}}evi{\'c}, Gary Froyland, Cecilia Gonz{\'a}lez-Tokman, and
  Sandro Vaienti, \emph{A spectral approach for quenched limit theorems for
  random hyperbolic dynamical systems}, Transactions of the American
  Mathematical Society \textbf{373} (2020), no.~1, 629--664.

\bibitem[DL18]{DeSimoiLiverani18}
Jacopo {De Simoi} and Carlangelo {Liverani}, \emph{{Limit theorems for
  fast-slow partially hyperbolic systems.}}, {Invent. Math.} \textbf{213}
  (2018), no.~3, 811--1016 (English).

\bibitem[DMT95]{DownMeynTweedie}
Douglas Down, Sean~P. Meyn, and Richard~L. Tweedie, \emph{Exponential and
  uniform ergodicity of {M}arkov processes}, Ann. Probab. \textbf{23} (1995),
  no.~4, 1671--1691.

\bibitem[{Dol}04a]{Dima04A}
Dmitry {Dolgopyat}, \emph{{Limit theorems for partially hyperbolic systems}},
  {Trans. Am. Math. Soc.} \textbf{356} (2004), no.~4, 1637--1689 (English).

\bibitem[{Dol}04b]{Dima04B}
\bysame, \emph{{On differentiability of SRB states for partially hyperbolic
  systems}}, {Invent. Math.} \textbf{155} (2004), no.~2, 389--449 (English).

\bibitem[Gou07]{gouezel2007statistical}
S{\'e}bastien Gou{\"e}zel, \emph{Statistical properties of a skew product with
  a curve of neutral points}, Ergodic Theory and Dynamical Systems \textbf{27}
  (2007), no.~1, 123--151.

\bibitem[GRS15]{GRS}
Stefano Galatolo, Jerome Rousseau, and Benoit Saussol., \emph{Skew products,
  quantitative recurrence, shrinking targets and decay of correlations.},
  {Ergodic Theory Dyn. Syst.} \textbf{35} (2015), no.~6, 1814--1845 (English).

\bibitem[Haf19]{hafouta2019}
Yeor Hafouta, \emph{Limit theorems for some skew products with mixing base
  maps}, Ergodic Theory and Dynamical Systems (2019), 1--31.

\bibitem[HM08]{HM08}
Martin {Hairer} and Jonathan~C. {Mattingly}, \emph{{Spectral gaps in
  Wasserstein distances and the 2D stochastic Navier-Stokes equations.}}, {Ann.
  Probab.} \textbf{36} (2008), no.~6, 2050--2091 (English).

\bibitem[HM11]{MM11}
Martin Hairer and Jonathan~C. Mattingly, \emph{Yet another look at {H}arris'
  ergodic theorem for {M}arkov chains}, Seminar on Stochastic Analysis, Random
  Fields and Applications VI (Basel) (Robert Dalang, Marco Dozzi, and Francesco
  Russo, eds.), Springer Basel, 2011, pp.~109--117.

\bibitem[HP06]{Pesin}
Boris {Hasselblatt} and Yakov {Pesin}, \emph{{Partially hyperbolic dynamical
  systems}}, {Handbook of dynamical systems. Volume 1B}, Amsterdam: Elsevier,
  2006, pp.~1--55 (English).

\bibitem[KKM20]{KorKoMel}
Alexey {Korepanov}, Zemer {Kosloff}, and Ian {Melbourne}, \emph{Explicit
  coupling argument for non-uniformly hyperbolic transformations}, Proc. R.
  Soc. Edinb., Sect. A, Math. (2020).

\bibitem[Klo20]{Kloeckner18}
Benoit Kloeckner, \emph{Extensions with shrinking fibers}, Ergodic Theory and
  Dynamical Systems (2020), no.~1-40.

\bibitem[{Liv}95]{Liverani95}
Carlangelo {Liverani}, \emph{{Decay of correlations}}, {Ann. Math. (2)}
  \textbf{142} (1995), no.~2, 239--301 (English).

\bibitem[MT93]{MeynTweedieBook}
S.P. Meyn and R.L. Tweedie, \emph{Markov chains and stochastic stability},
  Springer-Verlag, London, 1993.

\bibitem[NTV18]{nicol2018central}
Matthew Nicol, Andrew T{\"o}r{\"o}k, and Sandro Vaienti, \emph{Central limit
  theorems for sequential and random intermittent dynamical systems}, Ergodic
  Theory and Dynamical Systems \textbf{38} (2018), no.~3, 1127--1153.

\bibitem[OSY09]{OSY09}
William Ott, Mikko Stenlund, and Lai-Sang Young, \emph{Memory loss for
  time-dependent dynamical systems}, Mathematical Research Letters \textbf{16}
  (2009), no.~3, 463--475 (English).

\bibitem[PRH20]{GRH20}
Giulietti Paolo, Davide Ravotti, and Andy Hammerlindl, \emph{Quantitative
  global-local mixing for accessible skew products},
  https://arxiv.org/abs/2006.06539 (2020).

\bibitem[{Sam}16]{Sambarino13}
Mart\'{\i}n {Sambarino}, \emph{{A (short) survey on dominated splittings.}},
  {Mathematical congress of the Americas. First mathematical congress of the
  Americas, Guanajuato, M\'exico, August 5--9, 2013}, Providence, RI: American
  Mathematical Society (AMS), 2016, pp.~149--183 (English).

\bibitem[SB95]{Sh95}
Albert~Nikolaevich Shiryaev and Ralph.~Philip. Boas, \emph{Probability (2nd
  ed.)}, Springer-Verlag, Berlin, Heidelberg, 1995.

\bibitem[Sim12]{Simmons}
David Simmons, \emph{Conditional measures and conditional expectation;
  {R}ohlin's disintegration theorem}, Discrete \& Continuous Dynamical Systems
  - A \textbf{32} (2012), 2565.

\bibitem[Ste11]{stenlund2011non}
Mikko Stenlund, \emph{Non-stationary compositions of {A}nosov diffeomorphisms},
  Nonlinearity \textbf{24} (2011), no.~10, 2991.

\bibitem[Str14]{Stroock14}
Daniel~W. Stroock, \emph{An introduction to {M}arkov processes}, 2 ed.,
  Graduate Texts in Mathematics, vol. 230, Springer, 2014.

\bibitem[SW13]{santos2013distributional}
Sara~I Santos and Charles Walkden, \emph{Distributional and local limit laws
  for a class of iterated maps that contract on average}, Stochastics and
  Dynamics \textbf{13} (2013), no.~02, 1250019.

\bibitem[{Tsu}05]{Tsujii05}
Masato {Tsujii}, \emph{{Physical measures for partially hyperbolic surface
  endomorphisms}}, {Acta Math.} \textbf{194} (2005), no.~1, 37--132 (English).

\bibitem[TY20]{tanzi2020nonuniformly}
Matteo Tanzi and Lai-Sang Young, \emph{Nonuniformly hyperbolic systems arising
  from coupling of chaotic and gradient-like systems}, Discrete \&amp;
  Continuous Dynamical Systems-A \textbf{40} (2020), no.~10, 6015.

\bibitem[Via97]{SDDSViana}
Marcelo Viana, \emph{Stochastic dynamics of deterministic systems}, IMPA, 1997.

\bibitem[{Vil}09]{Villani09}
C\'edric {Villani}, \emph{{Optimal transport. Old and new.}}, vol. 338, Berlin:
  Springer, 2009 (English).

\bibitem[WW18]{walkden2018invariant}
CP~Walkden and Tom Withers, \emph{Invariant graphs of a family of non-uniformly
  expanding skew products over {M}arkov maps}, Nonlinearity \textbf{31} (2018),
  no.~6, 2726.

\bibitem[{You}02]{Young02}
Lai-Sang {Young}, \emph{{What are SRB measures, and which dynamical systems
  have them?}}, {J. Stat. Phys.} \textbf{108} (2002), no.~5-6, 733--754
  (English).

\bibitem[You19]{Young19}
Lai-Sang Young, \emph{Comparing chaotic and random dynamical systems}, Journal
  of Mathematical Physics \textbf{60} (2019), no.~5, 052701.

\end{thebibliography}

\end{document}